\documentclass[11pt,reqno]{amsart}

\usepackage{amsfonts,amscd,latexsym,amsmath,amssymb,enumerate,verbatim,amsthm,epsfig,alltt,cancel,hyperref,color}
\theoremstyle{plain}
\newtheorem{master}{Master}[section]
\newtheorem{prop}[master]{Proposition}
\newtheorem{thm}[master]{Theorem}
\newtheorem{fact}[master]{Fact}
\newtheorem{lem}[master]{Lemma}
\newtheorem{cor}[master]{Corollary}

\newtheorem{question}[master]{Question}

\newcommand{\Rea}{\mathbb{R}}
\newcommand{\ZZ}{\mathbb{Z}}
\newcommand{\UU}{\mathbb{U}}
\newcommand{\QU}{\mathbb{QU}}
\newcommand{\Nat}{\mathbb{N}}
\newcommand{\Rat}{\mathbb{Q}}
\newcommand{\Int}{\mathbb{Z}}
\newcommand{\Com}{\mathbb{C}}

\newcommand{\Iso}{\mbox{Iso}}
\newcommand{\Ur}{\mathbb{U}}
\newcommand{\Gu}{\mathbb{G}}
\newcommand{\age}{\mbox{Age}}

\newcommand{\KK}{\mathcal{K}}
\newcommand{\fra}{Fra\"\i ss\'e }
\newcommand{\kat}{Kat\v{e}tov }

\def\rest{\restriction}

\newcommand{\Rep}{\mathrm{Rep}}
\newcommand{\Aut}{\mathrm{Aut}}

\newcommand \NN{\mathbb{N}}

\newcommand{\Sub}{\mathrm{Sub}}
\newcommand{\Age}{\mathcal{C}}
\newcommand{\Stab}{\mathrm{Stab}}

\theoremstyle{definition}
\newtheorem{defin}[master]{Definition}
\theoremstyle{remark}
\newtheorem{rem}[master]{Remark}
\numberwithin{equation}{section}

\begin{document}
\title[Generic representations]{Generic representations of countable groups}
\author{Michal Doucha}
\author{Maciej Malicki}
\email{doucha@math.cas.cz, mamalicki@gmail.com}
\thanks{The first named author was supported by the GA\v CR project 16-34860L and RVO: 67985840.}
\address[M. Doucha]{Institute of Mathematics, Czech Academy of Sciences, \v Zitn\' a 25, 115 67 Praha 1, Czech republic}
\address[M. Malicki]{Department of Mathematics and Mathematical Economics, Warsaw School of Economics, al. Niepodleg\l o\' sci 162, 02-554 Warsaw, Poland}
\begin{abstract}
The paper is devoted to a study of generic representations (homomorphisms) of discrete countable groups $\Gamma$ in Polish groups $G$, i.e., elements in the Polish space $\Rep(\Gamma,G)$ of all representations of $\Gamma$ in $G$ whose orbit under the conjugation action of $G$ on $\Rep(\Gamma,G)$ is comeager. We investigate a closely related notion of finite approximability of actions on countable structures such as tournaments or $K_n$-free graphs, and we show its connections with Ribes-Zalesski-like properties of the acting groups. We prove that $\Int$ has a generic representation in the automorphism group of the random tournament (i.e., there is a comeager conjugacy class in this group). We formulate a Ribes-Zalesskii-like condition on a group that guarantees finite approximability of its actions on tournaments. We also provide a simpler proof of a result of Glasner, Kitroser and Melleray characterizing groups with a generic permutation representation.

We also investigate representations of infinite groups $\Gamma$ in automorphism groups of metric structures such as the isometry group $\Iso(\UU)$ of the Urysohn space, isometry group $\Iso(\UU_1)$ of the Urysohn sphere, or the linear isometry group $\mbox{LIso}(\mathbb{G})$ of the Gurarii space. We show that the conjugation action of $\Iso(\UU)$ on $\Rep(\Gamma,\Iso(\UU))$ is generically turbulent, answering a question of Kechris and Rosendal.
\end{abstract}
\maketitle
\section*{Introduction}

A \emph{representation} of a countable, discrete group $\Gamma$ in a Polish (i.e., separable and completely metrizable) topological group $G$ is a homomorphism of $\Gamma$ into $G$. The most frequently studied representations are finite-dimensional representations, i.e. homomorphisms into the matrix groups $\mathrm{GL}(n,\mathbb{K})$, where $n\in\Nat$ and $\mathbb{K}\in\{\Rea,\Com\}$, and unitary representations, i.e., homomorphisms into the unitary group $U(H)$ of separable Hilbert spaces $H$.

Especially within descriptive set theory, other interesting cases have been recently considered as well, e.g., representations in the isometry group $\Iso(\UU)$ of the Urysohn space (see \cite{MeTs}) or representations in the symmetric group $S_\infty$ of a countable set (see \cite{GKM}.) As a matter of fact, representations in automorphism groups of certain structures are nothing but actions on these structures, as is the case with the above mentioned examples, where representations are actions on finite-dimensional vector spaces,  Hilbert spaces, the Urysohn space $\mathbb{U}$ or a countable set (with no structure), respectively

Instead of a single representation, one can also investigate the space of all representations $\Rep(\Gamma,G)$, which can be equipped with a natural Polish topology. The action of $G$ on $\Rep(\Gamma,G)$ by conjugation, defined by
\[ (g.\pi)(\gamma)=g\pi(\gamma)g^{-1} \]
for $g \in G$, $\pi \in \Rep(\Gamma,G)$, $\gamma \in \Gamma$, leads to the concept of \emph{generic representation}, i.e., a representation whose orbit is comeager in $\Rep(\Gamma,G)$. As it turns out, such a global approach can provide insight into both the structure of $\Gamma$ and the structure of $G$, and it reveals interesting connections between the theory of countable groups, the theory of Polish groups, and model theory. 

An example of such a connection is the notion of \emph{ample generics}. A Polish group $G$ has ample generics if every free group $F_n$ on $n$ generators has a generic representation in $G$. Groups with ample generics satisfy certain very strong properties (see \cite{KeRo} for details), e.g., the automatic continuity property, which means that all (abstract) group homomorphisms from these groups into separable groups are continuous. On the other hand, ample generics are related to the \emph{Hrushovski property} that has been extensively studied in the context of \fra theory. Recall that a \fra class of structures $\KK$ has the Hrushovski property if for every $A \in \KK$ there is $B \in \KK$ containing $A$, and such that every partial automorphism of $A$ can be extended to an automorphism of $B$. It turns out that for a \fra class $\KK$ with sufficiently free amalgamation, the Hrushovski property implies the existence of ample generics in the automorphism group $\Aut(M)$ of its limit $M$, i.e., implies that $F_n$ has a generic representation in $\Aut(M)$ for every $n \geq 1$ (see \cite{HoHoLaSh} and \cite{KeRo}.)

Another aspect of this phenomenon has been revealed in the works of Herwig and Lascar \cite{HeLa} who showed that the Hrushovski property is closely related to the Ribes-Zalesskii property for free groups, which, in turn, is tied up with the profinite structure of countable groups. Later Rosendal \cite{Ro} proved that a countable group $\Gamma$ has the Ribes-Zalesskii property if and only if every action of $\Gamma$ on a metric space is \emph{finitely approximable}, i.e., every action of $\Gamma$ on a metric space $X$ can be approximated by actions of $\Gamma$ on finite metric spaces. It is not hard to see that for free groups the latter statement is equivalent to the Hrushovski property for metric spaces.

In the present paper, we continue this line of research. Although we work with representations in automorphism groups of structures from several areas of mathematics such as graphs, metric spaces, Banach spaces, etc., the paper is not just a collection of separate results. There is a unified methodology behind all our results. That is, whenever we construct some generic representation of some countable group $\Gamma$, it is a \fra limit of some `simple actions' of $\Gamma$. On the other hand, whenever we show that some group $\Gamma$ does not have a generic representation in an automorphism group of a structure of a given type, it is essentially by showing that there are too many actions of $\Gamma$ on structures of this given type,  i.e. one cannot construct a \fra limit of actions which would be dense. It is possible that this approach could be formalized and we partially do so for generic actions with finite orbits in the second section.

Let us present the main results of the paper. First of all, we study finite approximability for several classes of countable structures, namely tournaments, and $K_n$-free graphs. We show (Theorem \ref{prop:Z2nonapprox}) that no countable group that can be homomorphically mapped onto $\Int^2$ with a finitely generated kernel (in particular, $\Int^n$ for $n \geq 2$), has finite approximability on tournaments. This stands in sharp contrast with results of Rosendal \cite{Ro} who proved that each $\ZZ^n$ has finite approximability on graphs. Then we formulate a property closely related to the Ribes-Zalesskii property (Definition \ref{def_strong2RZ}), prove that it implies finite approximability on tournaments (Theorem \ref{thm_strong2RZ}), and verify it for $\Int$ (Proposition \ref{pr:Ztour2RZ}.)

Then we turn to triangle-free, and more generally $K_n$-free graphs, $n \geq 3$. We show (Theorem \ref{th:2RZKnFree} and Theorem \ref{th:3RZKnFree}) that the $2$-Ribes-Zalesskii property, and the $3$-Ribes-Zalesskii property, which are weak versions of the Ribes-Zalesskii property, form the lower and the upper `group-theoretic bounds' for finite approximability of actions on triangle-free graphs (resp. $K_n$-free graphs).

We also give in that section a simpler \fra-theoretic proof of the main result from \cite{GKM} that says that a countable, discrete group $\Gamma$ has a generic representation in $S_\infty$ if and only if it is solitary.

In the next section, we generalize known results relating, for finitely generated groups, finite approximability and generic representations with finite orbits. We prove (Theorem \ref{th:IndAmalg}) that for every sufficiently regular \fra class $\KK$ (to be more specific, for every \fra class with amalgamation allowing for amalgamating partial automorphisms), with \fra limit $M$, every action of a finitely generated group $\Gamma$ on $M$ is finitely approximable if and only if $\Gamma$ has a generic representation in $\Aut(M)$ with finite orbits. In particular, this theorem, combined with our results on tournaments, implies that there is a comeager conjugacy class in the automorphism group $\Aut(\mathcal{T})$ of the random tournament $\mathcal{T}$.

We also study generic representations in the automorphism groups of several metric structures, namely the Urysohn space $\UU$, the Urysohn sphere $\UU_1$ and the Gurarij space $\Gu$. We show (Theorem \ref{thm_meagerclassesUrysohn} and Theorem \ref{thm_meagerclassesGurarij}) that no countably infinite discrete group has a generic representation in $\Aut(X)$, where $X\in\{\UU,\UU_1,\Gu\}$. After the first version of this paper was written, we were informed by Julien Melleray that he had already proved the last result for the Urysohn space and sphere in his habilation thesis, published in \cite{Me}. The result for the sphere is however stated without a proof there, so we decided to publish it here (as well as the result for the whole Urysohn space, which is simpler). This allows us to only sketch the proof for the Gurarij space as it is analogous to that one for the sphere. This answers a question of Melleray from \cite{Me}. We moreover show additionally (Theorem \ref{thm_turbulent_element}) that the conjugation action of $\Iso(\UU)$ on $\Rep(\Gamma,\UU)$ is generically turbulent, for any infinite $\Gamma$. This answers a question of Kechris and Rosendal in \cite{KeRo} where they ask if the conjugacy action of $\Iso(\Ur)$ on itself is turbulent which is equivalent to the problem whether the conjugacy action of $\Iso(\Ur)$ on $\Rep(\Int,\Ur)$ is turbulent. The same ideas also work for the Urysohn sphere and the Gurarij space.

\section*{Terminology and basic facts}

Let $\KK$ be a countable, up to isomorphism, class of finite structures in a given language $\mathcal{L}$. We say that $\KK$ is a \emph{\fra class} if it has the hereditary property HP (for every $A\in \KK$, if $B$ can be embedded in $A$, then $B\in\KK$), the joint embedding property JEP (for any $A,B\in \KK$ there is $C\in \KK$ in which both $A$ and $B$ embed), and the amalgamation property AP (for any $A,B_1,B_2\in \KK$ and any embeddings $\phi_i\colon A\to B_i$, $i=1,2$, there are $C\in \KK$ and embeddings $\psi_i\colon B_i\to C$, $i=1,2$, such that $\psi_1\circ \phi_1=\psi_2\circ \phi_2$; we call any such $C$ an \emph{amalgam} of $B_1$ and $B_2$ over $A$). By a theorem due to Fra\"\i ss\'e, for every \fra family $\KK$, there is a unique up to isomorphism countable ultrahomogeneous structure $M$ (i.e., isomorphisms between finite substructures of $M$ extend to automorphisms of $M$) which embeds every member of $\KK$, and $\KK=\age(M)$, where $\age(M)$ is the class of all finite structures that can be embedded in $M$. In that case, we call $M$ the \emph{\fra limit} of $\KK$, see \cite[Section 7.1]{Hod}. As a matter fact, $M$ can be also characterized by its \emph{extension property}: a locally finite countable structure $X$ with $\age(X) \subseteq \KK$ is the \fra limit of $\KK$ if and only if for any $A,B \in \KK$, and embeddings $\phi:A \rightarrow X$, $i:A \rightarrow B$ there exists an embedding $\psi:B \rightarrow X$ such that $\psi \circ i =\phi$.

A locally finite structure $X$ such that $\age(X) \subseteq \KK$, for some class of finite structures $\KK$, is called a \emph{chain} from $\KK$. Suppose that $\alpha$ is an action by automorphisms of a group $\Gamma$ on a chain $X$ from a class $\KK$. We will say that $\alpha$ is \emph{finitely approximable} (by structures from $\KK$) if for every finite $F \subseteq \Gamma$ with $1 \in F$, and finite $X_0 \subseteq X$ there exists $Y \in \KK$, an action by automorphisms $\beta$ of $\Gamma$ on $Y$, and an injection $e:\alpha[F \times X_0] \rightarrow Y$ that embeds $\alpha$ restricted to $F$ and $X_0$ into $\beta$, i.e., 
\[ \beta(f,e(1,x))=e(1,\alpha(f,x)) \]
for $f \in F$, and $x \in X_0$. Note that when $\Gamma$ is the free group on finitely many generators then this corresponds to the well-studied problem when a tuple of finite partial automorphisms extends to a tuple of finite automorphisms. This was first proved by Hrushovski in \cite{Hr} for graphs and it remains open for tournaments.

Denote by $\mathcal{A}$ the class of all chains from $\KK$. We say that a \fra class $\KK$ has the \emph{Kat\v etov functor} if for any $X\in\mathcal{A}$ there are embeddings $\phi_X: X\hookrightarrow M$ and $\psi_X:\Aut(X)\hookrightarrow \Aut(M)$, where $M$ is the \fra limit of $\KK$, such that for every $f\in \Aut(X)$ and $x\in X$ we have
\[ \phi_X(f(x))=\psi_X(f)(\phi_X(x)).\]
Note that a prototypical example is the metric \fra class of metric spaces where it follows from the result of Uspenskij in \cite{Us} that this class has a Kat\v etov functor. See \cite{KuMa} for a reference on this topic.\\

Let $\Gamma$ be a countable discrete group, and let $G$ be a Polish group. The space $\Rep(\Gamma,G)$ of all representations of $\Gamma$ in $G$ (i.e., homomorphisms of $\Gamma$ into $G$) can be naturally endowed with a Polish topology by regarding it as a (closed) subspace of $G^\Gamma$. When $G$ is the automorphism group of some structure $X$, we usually write $\Rep(\Gamma,X)$ instead of the more precise $\Rep(\Gamma,\Aut(X))$. We will say that $\alpha \in \Rep(\Gamma,G)$ is a \emph{generic representation} if the orbit of $\alpha$ under the conjugation action of $G$ on $\Rep(\Gamma,G)$ is comeager in $\Rep(\Gamma,G)$. Note that in most cases we encounter the topological 0-1 law for representations is valid, i.e. either there is a generic representation in $\Rep(\Gamma,G)$, or all the conjugacy classes are meager. This is the case e.g. when there is a dense conjugacy class in $\Rep(\Gamma,G)$ (see Theorem 8.46 in \cite{Ke}).

\section{Finite approximability}

We recall that if $\Gamma$ is a discrete group, the \emph{profinite} topology on $\Gamma$ is the group topology on $\Gamma$ generated
by the basic open sets $gK$, where $g \in \Gamma$, and $K$ is a finite index normal subgroup of $\Gamma$. Thus, a subset $S \subseteq \Gamma$ is closed in the profinite topology on $\Gamma$ if for any $g \in \Gamma \setminus S$, there is a finite index normal subgroup $K \leq \Gamma$ such that $g \not \in SK$. Since this is a group topology, i.e., the group operations are continuous, $\Gamma$ is Hausdorff if and only if $\{1\}$ is closed, i.e., if for any $g \neq 1$ there is a finite index normal  subgroup $K$ not containing $g$. In other
words, $\Gamma$ is Hausdorff if and only if it is residually finite. A stronger notion than residual finiteness is {subgroup separability} or being LERF (\emph{locally extended residually finite}). Here a group $\Gamma$ is subgroup separable, or LERF, if any finitely generated subgroup $H \leq \Gamma$ is closed in the profinite topology on $\Gamma$. An even stronger notion is what we shall call the \emph{$n$-Ribes-Zalesskii} property, where $n \in \NN$, or $n$-RZ property  for brevity. Namely, for a fixed $n \in \NN$, a group $\Gamma$ is said to have the $n$-RZ property if any product $H_1H_2 \ldots H_n$ of finitely generated subgroups of $\Gamma$ is closed in the profinite topology on $\Gamma$. Finally, $\Gamma$ has the RZ property if it has the $n$-RZ property for every $n$. We refer to the paper \cite{RiZa} of Ribes and Zalesskii, where they prove that free groups have the RZ property.

In \cite{Ro}, Rosendal used the RZ property and its variants to characterize finite approximability of actions on metric spaces and graphs. In the following sections, we show that that this concept turns out to be useful also in studying $K_n$-free graphs (where $K_n$ is a clique on $n$ elements), and tournaments (recall that a tournament is a directed graph $(X,R)$ such that for any $x,y \in X$ exactly one of the arrows $(x,y)$, $(y,x)$ is in $R$.) 

Before proceeding further, let us mention a basic but very helpful fact that will be frequently used in this paper. For a group $\Gamma$ acting on a structure $X$ with a binary relation $R$, and for $x \in X$ with stabilizer $H_x$, if we have $R(x,g.x)$ for some $g\in\Gamma$, then we automatically have $R(x,gh.x)$ for any $h\in H_x$. Analogously, for $x,y \in X$ from distinct orbits, with stabilizers $H_x, H_y$, respectively, if we have $R(x,g.y)$ for some $g\in\Gamma$, then we automatically have $R(x,h_1gh_2.y)$ for any $h_1\in H_x$ and $h_2\in H_y$. These observations can be used to construct, or represent, actions of $\Gamma$ as actions by translation on disjoint unions of left cosets of the form $\Gamma/H_1\sqcup \ldots\sqcup \Gamma/H_n$, where each $H_i$ represents the stabilizer of a point.

\subsection{Tournaments}
In this section, we provide a general algebraic condition on a group which implies finite approximability on tournaments. We verify it for $\Int$, and leave it open for finitely generated non-abelian free groups. On the other hand, we show that a large class of groups containing $\Int^n$, $n \geq 2$, does not have finite approximability on tournaments. As it was mentioned in the introduction, this indicates that tournaments substantially differ from graphs because each $\Int^n$ has finite approximability on graphs (see \cite{Ro}).

Let us introduce the following definition.

\begin{defin}
Let $\Gamma$ be a countable group and let $H\leq \Gamma$. We say $H$ is \emph{good} if there are no $g\in \Gamma\setminus H$ and $h\in H$ such that $ghg\in H$.
\end{defin}

\begin{rem}\label{rem:goodsubgroups}
Clearly, if $\Gamma$ is abelian, then $H \leq \Gamma$ is good if and only if $\Gamma/H$ does not have elements of order $2$.
\end{rem}

The motivation for this definition comes from the following lemma.

\begin{lem}\label{lem_goodsubgrps}
Let $\Gamma$ be a countable group and let $H\leq \Gamma$. Then $H$ is good if and only if $H$ is a point-stabilizer for some action of $\Gamma$ on a tournament.
\end{lem}
\begin{proof}
Suppose $\Gamma$ acts on a tournament $(X,R)$ and  $H$ is the stabilizer of some $x\in X$. Suppose there are $g\in \Gamma\setminus H$ and $h\in H$ such that $g h g\in H$. Without loss of generality, assume that we have $R(x,g.x)$. Since $g h g\in H$, and the action of $\Gamma$ preserves the tournament relation $R$, we have 
\[ R(x,g.x)\Leftrightarrow R(g h g.x, g.x)\Leftrightarrow\] 
\[ R(h g.x,x)\Leftrightarrow R(g.x, x);\] 
a contradiction.\\

Conversely, assume that $H$ is good. We define a tournament structure on $\Gamma/H$, on which $\Gamma$ will act canonically. Using Zorn's lemma find a maximal subset $F\subseteq \Gamma$ satisfying that for no $f,g\in F$ is there $h\in H$ such that $f  h  g \in H$. Next we set $R(f  H, g H)$ if and only if there are $h_1,h_2\in H$ and $g'\in F$ such that $f^{-1} g=h_1 g' h_2$. Clearly, the action of $\Gamma$ preserves the relation $R$, so we must check that it is a tournament relation. Suppose there is $g\in \Gamma\setminus H$ such that both $R(H,g H)$ and $R(g H, H)$ hold true. Then also $R(H, g^{-1} H)$ holds, so there are $h_1,h_2,h_3,h_4\in H$ such that $h_1 g h_2, h_3 g^{-1} h_4\in F$. This clearly violates the condition imposed on $F$ since $(h_1 g h_2)h_2^{-1} h_3^{-1}(h_3 g^{-1} h_4)\in H$. So suppose now that there is $g\in \Gamma\setminus H$ such that neither $R(H,g H)$ nor $R(H, g^{-1} H)$ hold true. Then we claim we may add $g$ into $F$ contradicting the maximality of $F$. Indeed, suppose that by adding $g$ into $F$ we violate the condition imposed on $F$. Since $H$ is good, we cannot have that $g h g\in H$ for some $h\in \Gamma$. So there are $f\in F$ and $h\in H$ such that $g h f\in H$ or $f h g\in H$. Assume the former case, the latter one is dealt with analogously. Then we have $g h f=h'$ for some $h'\in H$, so $g^{-1}=h f {h'}^{-1}$, so $R(H,g^{-1} H)$; a contradiction.
\end{proof}

We record some basic properties of good subgroups.

\begin{lem}\label{lem_goodsubgrps2}
Let $\Gamma$ be a countable group. We have
\begin{enumerate}
\item\label{lem_item1} If $(H_i)_{i\in I}$ is a collection of good subgroups of $\Gamma$, then $\bigcap_{i\in I} H_i$ is also good.
\item\label{lem_item1*} If $H\leq\Gamma$ is a good subgroup, then any conjugate $xHx^{-1}$ is also good.
\item\label{lem_item2} If $H\leq \Gamma$ is a good subgroup, then the maximal normal subgroup of $\Gamma$ contained in $H$ is also good. In particular, if $H$ is a good subgroup of finite index, then there is a good normal subgroup of finite index.

\end{enumerate}

\end{lem}
\begin{proof}
\eqref{lem_item1} Suppose $(H_i)_{i\leq I}\leq \Gamma$ are good. Set $H=\bigcap_{i\in I} H_i$ and take some $g\in \Gamma\setminus H$ such that there is $h\in H$ with $g h g\in H$. There exists $i\in I$ such that $g\notin H_i$. Then however $g h g\in H\leq H_i$ which violates that $H_i$ is good.\\

\eqref{lem_item1*} Let $H\leq \Gamma$ be good. Suppose that for some $x\in\Gamma$ and $g\notin xHx^{-1}$ we have $g x H x^{-1}\cap x H x^{-1}\neq\emptyset$. Then $x^{-1} g x H x^{-1} g x\cap H\neq \emptyset$, which is a contradiction since $H$ is good and $x^{-1} g x\notin H$.\\

\eqref{lem_item2}: Suppose that $H\leq \Gamma$ is a good subgroup. The maximal normal subgroup of $\Gamma$ contained in $H$ is equal to $\bigcap_{g\in\Gamma} g^{-1} H g$, so we are done by \eqref{lem_item1} and \eqref{lem_item1*}.
\end{proof}
\begin{question}\label{quest_goodgrp}
Is it true that for every countable group $\Gamma$, if $H\leq \Gamma$ is a good subgroup and $N\leq \Gamma$ is a good normal subgroup, then the subgroup $H N$ is good?
\end{question}

We now proceed with the non-approximability result.
\begin{thm}\label{prop:Z2nonapprox}
Let $\Gamma$ be a countable group such that there exists a finitely generated normal subgroup $N\leq \Gamma$ with $\Gamma/N\cong \Int^2$. Then $\Gamma$ does not have finite approximability on tournaments.

In particular, with the exception of $\Int$, no finitely generated nilpotent torsion-free group has finite approximability on tournaments.
\end{thm}
\begin{proof}
First we treat the case when $\Gamma$ is actually equal to $\Int^2$. Let us define the following action of $\Int^2$ on a tournament $R$. Denote the two canonical generators of $\Int^2$ by $a$ and $b$. The tournament will have two infinite orbits with base points $x$ and $y$. The stabilizer $H_x$ of $x$ will be equal to $\langle b\rangle$, and the stabilizer $H_y$ of $y$ will be equal $\langle 2a-b\rangle$. Therefore the set of vertices of the tournament can be identified with the set $\Int^2/H_x \sqcup \Int^2/H_y$, and $\Int^2$ will act on it by translation.

We define the tournament on the first orbit as follows: we set $R(x,(na).x)$, for $n>0$, and $R((na).x,x)$ for all $n<0$. Since $\Int^2/\langle b\rangle$ is isomorphic to $\langle a'\rangle$, where $a'$ is the projection of $a$ on the quotient, this is easily checked to give rise to a tournament relation on the orbit of $x$ that is invariant under the action of $\Int$.

Also, the quotient $\Int^2/\langle 2a-b\rangle$ is isomorphic to $\langle a'\rangle$, where $a'$ is the projection of $a$ on the quotient. It follows that, as before, we can define an invariant tournament relation on the second orbit by setting $R(y,(na).y)$, for $n>0$, and $R(na).y,y)$, for $n<0$.

Finally, since $\Int^2/(H_x+H_y)\cong \Int_2$, where the only non-zero element $a''$ of order two is the projection of $a$ on the quotient, to define the tournament for elements from different orbits, it suffices to set $R(y,x)$, and so
$$R((h_x+h_y).y,(h'_x+h'_y).x)), \  R((a+h_x+h_y).y,(a+h'_x+h'_y).x)),$$
for $h_x,h'_x \in H_x$, $h_y,h'_y \in H_y$, and to set $R(x,a.y)$, and so
$$R((h_x+h_y).x,a+(h'_x+h'_y).y), \  R((a+h_x+h_y).x,(h'_x+h'_y).y)),$$
for $h_x,h'_x \in H_x$, $h_y,h'_y \in H_y$.\\

We claim that this tournament is not finitely approximable. To be more specific, we claim that it is not possible to finitely approximate the two finitely generated stabilizers of $x$ and $y$, and the relations $R(y,x)$, $R(x,a.y)$.

Suppose otherwise that there is such a finite approximation, with the tournament relation denoted by $S$. Its stabilizers $K_x$, $K_y$ of $x$, $y$, respectively, contain $H_x$, $H_y$, respectively, so $K_x$ contains $b$. We claim that the projection of $a$ to the quotient $\Int^2/K_x$ has finite odd order. Otherwise, by Remark~\ref{rem:goodsubgroups}, $K_x$ is not good, and so, by Lemma~\ref{lem_goodsubgrps}, it cannot be a point-stabilizer for an action on a tournament. Therefore $K_x$ contains $na$ for some odd $n\in\Nat$. It follows that the projection of $a$ to the quotient $\Int^2/( K_x+K_y)$ must be trivial since it has odd order in $\Int^2/K_x$, where $K_x\subseteq K_x+K_y$, and even order in $\Int^2/(H_x+H_y)$, where again $H_x+H_y\subseteq K_x+K_y$; in other words, $a=(n+1)a-na=\frac{n+1}{2}b-na\in H_y+K_x\subseteq K_x+K_y$. Then, however, $S$ cannot satisfy $S(y,x)$ and $S(x,a.y)$.\\

In the general case, when $\Gamma$ is such that there is a finitely generated normal subgroup $N\leq \Gamma$ with $\Gamma/N\cong \Int^2$, we define an action on exactly the same tournament as above. The action of $\Gamma$ factorizes through the action of $\Int^2$, i.e., its kernel contains $N$. Since $N$ is finitely generated, we again reach a contradiction with finite approximability.
\end{proof}
The aim of the next definition is to isolate a Ribes-Zalesskii-like property of groups that would guarantee finite approximability of tournaments - in the same way as the 2-Ribes-Zalesskii property guarantees finite approximability for graphs. We shall call it tournament 2-Ribes-Zalesskii property as it is apparently very similar to the 2-Ribes-Zalesskii property, however not obviously equivalent, or weaker or stronger.
\begin{defin}\label{def_strong2RZ}
We say that $\Gamma$ has the \emph{tournament 2-RZ property} if
for any $g_i \in \Gamma$,  $i \leq n$, finitely generated good subgroups $K_i, H_i \leq \Gamma$, $i \leq n$, such that $g_i \not \in K_iH_i$, and any finitely generated good subgroups $M_j \leq \Gamma$, $j \leq m$, there exists a finite index normal subgroup $N\leq \Gamma$ such that

\begin{itemize}
\item $g_i \not \in K_i H_i N$ for each $i\leq n$ (in particular, $\Gamma$ has the 2-RZ property for good subgroups),
\item $M_j N$ is good for each $j\leq m$.
\end{itemize}

\end{defin}

\begin{thm}\label{thm_strong2RZ}
Suppose that $\Gamma$ is a finitely generated group with the tournament 2-RZ property. Then every action of $\Gamma$ on a tournament is finitely approximable.

Conversely, if every action of $\Gamma$ on a tournament is finitely approximable, and, moreover, Question \ref{quest_goodgrp} has a positive answer for $\Gamma$, then $\Gamma$ has the tournament 2-RZ property.
\end{thm}
Before proceeding to the proof, we verify this condition for $\Int$.
\begin{prop}
\label{pr:Ztour2RZ}
The group $\Int$ has the tournament 2-RZ property, and so $\Int$ has finite approximability on tournaments.
\end{prop}
\begin{proof}
Fix good subgroups $K_i,H_i \leq \Int$, $i\leq n$, $M_j \leq \Int$, $j \leq m$, and fix $g_i \in \Int$, $i \leq n$, such that $g_i \notin K_i+H_i$. It follows that each $M_j$ is of the form $c_j\Int$, where $c_j \geq 0$ is odd  or equal to $0$, and also each $K_i+H_i$ is of the form $d_i\Int$, where $d_i \geq 0$ is odd or equal to $0$. We choose $c>1$ which is a multiple of all non-zero $c_j$, $j\leq m$, and non-zero $d_i$, $i\leq n$, and, moreover, is strictly greater than all $|g_i|$, $i\leq n$. We set $N=c\Int$. It is straightforward to verify that each $M_j+N$ is equal to either $M_j$ or $N$ (if $c_j=0$), and so it is good. Similarly, each $K_i+H_i+N$ is equal to either $K_i+H_i$ or $N$ (if $d_i=0$.) Since $c>|g_i|$, in both cases we get that $g_i\notin K_i+H_i+N$.
\end{proof}

Without loss of generality, we can assume that no $M_j$ is equal to $\Int$ or the trivial subgroup.

\begin{proof}[Proof of Theorem \ref{thm_strong2RZ}]
Let $\Gamma$ be a finitely generated group with the tournament 2-RZ property, and suppose that $\Gamma$ acts on a tournament $(X,R)$. Take now some finite partial subaction $\alpha$ of $\Gamma$ on $X$. We may suppose that $\alpha$ is given by the following data:
\begin{itemize}
\item finitely many orbits $O_1,\ldots,O_m$, each $O_j$ with some base point $x_j$ and a finitely generated stabilizer $M_j$ of $x_j$;
\item for each $j\leq m$ a finite subset $F_j\subseteq \Gamma$ such that for each $f\in F_j$ we have $R(x_j, f.x_j)$;
\item for every $i\neq j\leq m$ we have a finite subset $F_{i,j}\subseteq \Gamma$ such that for every $f\in F_{i,j}$ we have $R(x_i,f.x_j)$.

\end{itemize}
Moreover, we may and will assume that for every $j\leq m$, $F_j$ intersects each left coset of $M_j$ in at most one element. This is clear since if $f,f'$ lie in the same left coset of $M_j$ and $R(x_j,f.x_j)$, then automatically $R(x_j,f'.x_j)$. Analogously, we will assume that for every $i\neq j\leq m$, each $F_{i,j}$ intersects each double coset $M_i\backslash \Gamma/M_j$ in at most one element.

Notice that for every $j\leq m$ and any two elements $f,g\in F_j$ we have 
\[ f M_j g\cap M_j=\emptyset,\]
and for every $i\neq j\leq m$ and $f\in F_{i,j}$ and $g\in F_{j,i}$ we have
\[f M_j\cap M_i g^{-1}=\emptyset.\]

We show the former, the latter is proved analogously. Suppose that $f h_1 g=h_2$, for some $h_1,h_2\in M_j$, i.e. $g=h f^{-1} h'$, for some $h,h'\in M_j$. We have $R(x_j,g.x_j)$ and $R(f^{-1}.x_j,x_j)$; thus, since $h$ and $h'$ fix $x_j$, we have $R(h f^{-1} h'.x_j,x_j)$ which implies $R(g.x_j,x_j)$, and that is a contradiction.

Notice that the conditions $f M_j g\cap M_j=\emptyset$, resp. $f M_j\cap M_i g^{-1}=\emptyset$ are equivalent to the conditions $gf\notin (g M_j g^{-1}) M_j $, resp. $gf\notin (g M_i g^{-1}) M_j$. Therefore, for every $j\leq n$ and every pair $f,g\in F_j$, if we set $g_{j,f,g}=gf$, $K_{j,f,g}=g M_j g^{-1}$ and $H_{j,f,g}=M_j$, we get $g_{j,f,g}\notin K_{j,f,g} H_{j,f,g}$. Analogously, for every $i\neq j\leq m$ and $f\in F_{i,j}$ and $g\in F_{j,i}$, if we set $g_{i,j,f,g}=gf$, $K_{i,j,f,g}=g M_i g^{-1}$ and $H_{i,j,f,g}=M_j$, we get $g_{i,j,f,g}\notin K_{i,j,f,g} H_{i,j,f,g}$. We enumerate all these triples as $(g_i,K_i,H_i)_{i\leq n}$ and by applying the tournament 2-RZ property (since by Lemmas \ref{lem_goodsubgrps} and \ref{lem_goodsubgrps2} all the subgroups involved are good) we can find a finite index normal subgroup $N$ satisfying:
\begin{enumerate}
\item\label{it1} for every $j\leq m$ and every $f,g\in F_j$ we have 
\[ f M_j g\cap M_j N=\emptyset;\]
\item\label{it2} for every $i\neq j\leq m$ and every $f\in F_{i,j}$ and $g\in F_{j,i}$ we have 
\[M_i f M_j \cap M_i g^{-1} M_j N=\emptyset\]
(this is equivalent to $g f\notin (g M_i g^{-1}) M_j N$);

\item\label{it3} for every $j\leq m$, $M_j N$ is good.\\

\end{enumerate}

Now we define a finite tournament extending this finite fragment. The underlying set is
\[ Y=\Gamma/(M_1 N)\sqcup \ldots \sqcup \Gamma/(M_m N),\]

and the action $\beta$ of $\Gamma$ on it is the natural one by left multiplication.
For every $j\leq m$, let $F'_j\subseteq \Gamma$ be a finite subset, which intersects each left coset of $M_j N$ in at most a singleton, satisfying
\begin{itemize}
\item $F_j\subseteq F'_j$ (extension);
\item for every $f,g\in F'_j$ we have
\[ f M_j g\cap M_j N=\emptyset \mbox{ (consistency)};\]
\item for every $f\in \Gamma\setminus (F'_j M_j N)$ there exist $g\in F'_j$ and $h\in M_j N$ such that either $f h g\in M_j N$ or $g h f\in M_j N$ (maximality).

\end{itemize}
It is possible to find such a set by \eqref{it1} and \eqref{it3}. Indeed, by \eqref{it1} we have that that for every $f,g\in F_j$, 
\[ f M_j g \cap M_j N=\emptyset. \]
We want to extend $F_j$ as much as possible still satisfying this consistency property. Suppose that $F'_j$ is a maximal subset (intersecting each left coset of $M_j N$ in at most singleton) containing $F_j$ and satisfying that for every $f,g\in F'_j$ we have 
\[ f M_j g M_j\cap M_j N=\emptyset.\]
Suppose the third condition is not satisfied, i.e. there is $f\in \Gamma\setminus (F'_j M_j N)$ such that for every $g\in F'_j$ and $h\in M_j N$ neither $f h g\in M_j N$, nor $g h f\in M_j N$. Then we claim we may extend $F'_j$ by $f$ contradicting its maximality. Suppose that it is not the case. It means that adding $f$ violated the consistency condition, i.e., $f M_j g\cap M_j N\neq \emptyset$ or $g M_j f\cap M_j N\neq \emptyset$, for some $g\in F'_j\cup\{f\}$. By our assumption, $g$ is not from $F'_j$, so necessarily $g=f$, and  $f h f\in M_j N$ for some $h\in M_j N$. This, however, contradicts \eqref{it3}. Clearly, $F'_j$ is finite, since $M_jN$ has finite index.\\

Analogously, for each $i\neq j\leq m$, let $F'_{i,j}\subseteq \Gamma$ be a finite subset, which intersects each double coset $M_iN\backslash \Gamma/ M_jN$ in at most singleton, satisfying
\begin{itemize}
\item $F_{i,j}\subseteq F'_{i,j}$ (extension);
\item for every $i\neq j\leq m$ and every $f\in F'_{i,j}$ and $g\in F'_{j,i}$ we have
\[ M_i f M_j\cap M_i g^{-1} M_j N=\emptyset \mbox{ (consistency)};\]
\item for every $f\in \Gamma$ there are $h_i\in M_i N$ and $h_j\in M_j N$ such that either $h_i f h_j\in F'_{i,j}$ or $h_j f^{-1} h_i\in F'_{j,i}$ (maximality).

\end{itemize}
It is possible to find such sets $F'_{i,j}$ as follows. By \eqref{it2} we have that $F_{i,j}$ satisfies the consistency condition above, so we again want to extend $F_{i,j}$ as much as possible. Suppose that $i>j$ then we set $F'_{i,j}=F_{i,j}$. So suppose that $i<j$. Then for each double coset\\ ${(M_i N \backslash \Gamma/ M_j N)\setminus ((\bigcup_{f\in F_{i,j}} M_i N f M_j N)\cup(\bigcup_{g\in F_{j,i}} M_i N g^{-1} M_j N))}$ we choose some representative and put it into $F'_{i,j}$.

Let us now show that such $F'_{i,j}$ satisfies the consistency and maximality conditions. It follows from the symmetricity in definitions that verifying consistency, resp. maximality for $F'_{i,j}$, verifies it also for $F'_{j,i}$. So we may suppose that $i<j$. To prove the maximality, choose some $f\in\Gamma$ and suppose that there are no $h_1\in M_i N$ and $h_2\in M_j N$ such that $h_2 f^{-1} h_1\in F_{j,i}=F'_{j,i}$. Then by the definition of $F'_{i,j}$ we put some representative of the double coset $M_i N f M_j N$ into $F'_{i,j}$ which proves the maximality condition. Now we show consistency. If $F'_{i,j}$ were not consistent, there would be $f\in F'_{i,j}$ and $g\in F'_{j,i}=F_{j,i}$ with $M_i f M_j\cap M_i g^{-1} M_j N\neq\emptyset$. The element $f$ cannot be from $F_{i,j}$ since that would contradict \eqref{it2}; so $f\in F'_{i,j}\setminus F_{i,j}$, which is also a contradiction since $f$ was chosen not to be from the double coset $M_i g^{-1} M_j$. \\

Now we define the tournament relation $S$ on $Y$. For any $j\leq m$ and any $f,g\in \Gamma$ we set $S(f M_j N, g M_j N)$ if and only if for some $h\in M_j N$ we have $f^{-1} g h\in F'_j$. We claim that exactly one of the options $S(f M_j N, g M_j N)$ and $S(g M_j N, f M_j N)$ happens. To simplify the notation, we show that for any $f\in \Gamma$ either $S(M_j N, f M_j N)$ or $S(f M_j N,M_j N)$ happens. First we show that at least one of the options happens, then that at most one of them happens.

If $f\in F'_j M_j N$, then $S(M_j N, f M_j N)$ by definition. So suppose that $f\in \Gamma\setminus (F'_j M_j N)$. Then by the maximality condition there exist $g\in F'_j$ and $h\in M_j N$ such that either $f h g\in M_j N$ or $g h f\in M_j N$. Suppose the former. Then we have $$S(M_j N, g M_j N)\Leftrightarrow S(f h M_j N,f h g M_j N)\Leftrightarrow S(f M_j N, M_j N).$$  The latter condition is treated analogously.

Now we show that at most one of the conditions happens. Suppose on the contrary that both $S(M_j N, f M_j N)$ and $S(f M_j N,M_j N)$ hold. Then however $S(M_j N, f^{-1} M_j N)$ holds, therefore there are $h_1,h_2\in M_j N$ such that $f h_1, f^{-1} h_2\in F'_j$. This violates the consistency condition though as we have 
\[ f h_1 M_j N f^{-1} h_2 M_j N= M_j N.\]

Now for $i\neq j\leq m$ we set $S(f M_i N, g M_j N)$ if there are $h_i\in M_i$, $h_j\in M_j$ and $g'\in F'_{i,j}$ such that $f^{-1} g=h_i g' h_j$. We again claim that exactly one of the options  $S(f M_i N, g M_j N)$ and  $S(g M_j N, f M_i N)$ happens. Again it suffices to check that for $f\in \Gamma$ exactly one of the options  $S(M_i N, f M_j N)$ and  $S(f M_j N, M_i N)$ happens. First we show that if at least one of the options happens, then that at most one of them happens.

By the maximality condition there are $h_i\in M_i$ and $h_j\in M_j$ such that either $h_i f h_j\in F_{i,j}$ or $h_j f^{-1} h_i\in F_{j,i}$. In the first case we clearly have that $S(M_i N, f M_j N)$ holds true, while in the latter case we have $S(M_j N, f^{-1} M_i N)$, therefore $S(f M_j N, M_i N)$.

Suppose now that both $S(M_i N, f M_j N)$ and  $S(f M_j N, M_i N)$ hold true. Then there are $h_i,h_j,h'_i,h'_j$ such that $h_i f h_j\in F'_{i,j}$ and $h'_j f^{-1} h'_i\in F'_{j,i}$. Then it clearly violates the consistency condition above since
\[ M_i h_i f h_j M_j\cap M_i (h')^{-1}_i f (h')^{-1}_j M_j N\neq \emptyset.\]

The desired partially defined embedding $e:X\rightarrow Y$ sends for each $j\leq m$: $x_j$ to $M_jN$ viewed as a base point of an orbit in $Y=\Gamma/(M_1N) \sqcup \ldots \sqcup \Gamma/(M_mN)$, and $g.x_j$ to $gM_jN$, for $g\in F_j$ and $g\in F_{i,j}$ where $i\neq j\leq m$. This clearly satisfies what is needed.
\bigskip

Now we prove the reverse implication. We suppose that every action of $\Gamma$ on a tournament is finitely approximable and moreover that Question \ref{quest_goodgrp} has a positive answer for $\Gamma$. We show that $\Gamma$ has the tournament 2-RZ property.

Take some finite number of triples $(g_1,K_1, H_1)$,\ldots, $(g_n,K_n, H_n)$ where $g_i\notin K_i H_i$ and $K_i$ and $H_i$ are finitely generated good subgroups, for $i\leq n$, and some finite number of finitely generated good subgroups $M_1,\ldots, M_m$. We define an action of $\Gamma$ on a tournament $(X,R)$. The underlying set $X$ will be the disjoint union (of orbits) 
\[ (\bigsqcup_{i\leq n} \Gamma/K_i \sqcup \Gamma/H_i)\sqcup \bigsqcup_{j\leq m} \Gamma/M_j,\]
and the action of $\Gamma$ is the canonical one. On each orbit of $X$ we define a tournament structure using Lemma \ref{lem_goodsubgrps}. Now the arrows between different orbits are defined arbitrarily just to satisfy that for all $i\leq n$ we have $R(K_i, g_i H_i)$ and $R(H_i, K_i)$. 

We want to finitely approximate this tournament action so that the stabilizers of the orbits are preserved and so that for each $i\leq n$ the relations $R(K_i, g_i H_i)$ and $R(H_i, K_i)$ are preserved. Therefore we get an action of $\Gamma$ on a finite tournament $(Y,S)$ with $2n+m$ orbits with stabilizers $K_i\leq K'_i$, $H_i\leq H'_i$, for $i\leq n$, and $M_j\leq M'_j$, for $j\leq m$, i.e. we may view $Y$ as
\[ Y=(\bigsqcup_{i\leq n} \Gamma/K'_i \sqcup \Gamma/H'_i) \sqcup \bigsqcup_{j\leq m} \Gamma/M'_j \]
with the natural action of $\Gamma$. Moreover, we have for all $i\leq n$, $S(K'_i, g_i H'_i)$ and $S(H'_i, K'_i)$, which implies that that for every $i\leq n$, $g_i\notin K'_i H'_i$. Since all the stabilizers are good subgroups of finite index, using Lemma \ref{lem_goodsubgrps2} we can find good normal finite index subgroups $H''_i\leq H'_i$, $K''_i\leq K'_i$, for $i\leq n$, and $M''_j\leq M'_j$, for $j\leq n$. Again using Lemma \ref{lem_goodsubgrps2} we get that the intersection $N$ of all these good normal finite index subgroups is again a good normal finite index subgroup. Clearly, for every $i\leq n$ we have $g_i\notin K_i H_i N$. Now if the answer of Question \ref{quest_goodgrp} is positive for $\Gamma$, then $M_j N$ is good for every $j\leq m$, therefore $N$ is as desired, and we are done.
\end{proof}

\begin{question}
Do finitely generated free groups have the tournament 2-Ribes-Zalesskii property?
\end{question}

In our opinion, Proposition \ref{prop:Z2nonapprox} implying the lack of the tournament 2-RZ property for finitely generated abelian free groups $\Int^d$, $d\geq 2$, does not necessarily suggest that the answer to the above question is negative. In a somewhat related research, Kwiatkowska (\cite{Kw}) showed that finitely generated free groups have a generic representation in the Cantor space, while Hochman (\cite{Ho}) proved that $\Int^d$, $d\geq 2$, do not have one.


\subsection{$K_n$-free graphs}
Now we turn to triangle-free graphs. Using techniques from \cite{Ro}, we prove that the $2$-RZ and $3$-RZ properties are the lower and the upper bounds for finite approximability of actions on triangle-free graphs (and on $K_n$-free graphs).

\begin{thm}
\label{th:3RZKnFree}
Let $\Gamma$ be a countable group satisfying the $3$-RZ property. Then every action of $\Gamma$ on a triangle-free graph is finitely approximable. More generally, every action of $\Gamma$ on a $K_n$-free graph, for $n \geq 3$, is finitely approximable.
\end{thm}

\begin{proof}
Fix an action $\alpha$ of $\Gamma$ on a triangle-free graph $X$ identified with a metric space $(X,d)$ with possible values of $d(x,y)$, $x \neq y \in X$, either $1$, if there is an edge between $x$ and $y$, or $2$ otherwise.  
For $x \in X$, let $H_x \leq \Gamma$ be the stabilizer of $x$. Observe that the assumption that $X$ is triangle-free means that  
\[ d(x,f_1.y)=d(x,f_2.z)=d(y,f_1^{-1}f_2.z)=1. \]
does not hold for any $x,y,z \in X$, and $f_1,f_2 \in \Gamma$. Moreover, distances must be constant on appropriate double cosets, i.e., for every $g_1 \in H_x f_1 H_y$, $g_2 \in H_x f_2 H_z$,
\[ d(x,g_1.y)=d(x,g_2.z)=d(y,g_1^{-1}g_2.z)=1. \]
does not hold either. But this is equivalent to saying that for every $f_1, f_2, f_3 \in \Gamma$ such that 
\[ d(x,f_1.y)=d(x,f_2.z)=d(y,f_3.z)=1 \]
we have
\[ H_y f_1 H_x f_2 H_z \cap H_y f_3 H_z= \emptyset, \]
or that
\[ H_y f_1 H_x f_2 H_z \cap \{f_3\}=\emptyset. \]
Actually, the above formula can be also written as
\[  f_1 f_2(f_2^{-1} f_1^{-1} H_y f_1 f_2)(f_2^{-1}H_x f_2) H_z \cap \{f_3\}=\emptyset. \]

Fix finite $F \subseteq \Gamma$ with $1 \in F$, and $A \subseteq X$. Without loss of generality, we can assume that no two elements of $A$ are in the same orbit under $\alpha$. Because $\Gamma$ has the $3$-RZ property, we can find a finite-index normal subgroup $K \leq \Gamma$ such that for every $f_1, f_2, f_3 \in F$, and $x,y,z \in A$, with 
\[ d(x,f_1.y)=d(x,f_2.z)=d(y,f_3.z)=1 \]
we have
%
%
%
\[  H_y f_1 H_x f_2 H_z \cap H_y f_3 H_z K = \emptyset, \]
and also, for every $f \in F$ and $x \in A$ with $f \not \in H_x$, 
\[ H_x \cap fK= H_xK \cap \{f\}  =\emptyset. \]

Let $L_x=H_xK$ for $x \in A$. We define a finite graph $Y= \coprod_{x \in A} \Gamma/L_x$ by specifying a metric $\rho$ on $Y$, with values in $\{0,1,2\}$, so that, for $fL_x \neq gL_y$,
\[ \rho(f L_x,g L_y)=1 \]
if and only if $f^{-1}g \in H_x g' H_yK$ for some $g' \in F$ with $d(x,g'.y)=1$.

Note first that $\rho$ is trivially a metric because the triangle inequality is satisfied for any mapping from $Y \times Y$ into $\{0,1,2\}$ that does not assign $0$ to pairs of the form $(y,y')$, $y \neq y'$. Also, it is clearly invariant under the left-translation action $\beta$ of $\Gamma$ on $Y$. We need to verify that $Y$ is triangle-free, and that $\alpha$, when restricted to $F$ and $A$, embeds into $\beta$ via the mapping $f.x \mapsto fL_x$, for $f \in F$, $x \in A$.

In order to see that $Y$ is triangle-free, suppose the contrary, and fix $x,y,z \in A$ and $g_1,g_2 \in \Gamma$ such that
\[ \rho(L_x,g_1 L_y)=\rho(L_x, g_2 L_z)=\rho(g_1L_y,g_2L_z)=1. \]
But this means that
\[ g_1 \in H_x f_1 H_yK, \, g_2 \in H_x f_2 H_yK, \, g_1^{-1}g_2 \in H_yf_3H_zK, \]
where $f_1, f_2, f_3 \in F$ are such that
\[ d(x,f_1.y)=d(x,f_2.z)=d(y,f_3.z)=1. \]
Since we have that
\[ KH_yf_1Hxf_2H_zK \cap H_yf_3H_zK=H_yf_1Hxf_2H_z \cap H_yf_3H_zK= \emptyset, \]
for any such $f_1,f_2,f_3$, this is impossible.

In a similar fashion, we can show that $f.x \mapsto fL_x$ is a mapping embedding $\alpha$ restricted to $F$ and $A$ into $\beta$.

The general statement for $K_n$-free graphs, $n \geq 3$, can be proved in the same way. For example, for $n=4$, the condition that $X$ does not contain $K_4$ means that, for any fixed set $X_0=\{x,y,z,w\} \subseteq X$ of $4$ elements coming from pairwise distinct orbits of $\alpha$, and $f_1,f_2,f_3 \in \Gamma$
\[ d(x,f_1.y)=d(x,f_2.z)=d(x,f_3.w)=d(y,f_1^{-1}f_2.z)=d(y,f_1^{-1} f_3.w)=d(z,f_2^{-1} f_3 w)=1 \]
does not hold. This can be rewritten as follows: for any $f_{a,b} \in \Gamma$, $a \neq b \in X_0$ such that $d(a, f_{a,b}.b)=1$, at least one of the intersections
\[ H_b f^{-1}_{a,b} H_a f_{a,c} H_c \cap H_b f_{b,c} H_c, \]
is empty. Now, because $\Gamma$ has the $3$-RZ property, we can construct a $K_4$-free graph $Y$ extending $X$ exactly as above.

\end{proof}

\begin{thm}
\label{th:2RZKnFree}
Let $\Gamma$ be a countable group whose actions on triangle-free graphs are finitely approximable. Then $\Gamma$ has the $2$-RZ property.
\end{thm}

\begin{proof}
Fix finitely generated subgroups $H_1, H_2$ of $\Gamma$, and $g \in \Gamma \setminus H_1H_2$. We need to show that there is a finite index normal subgroup $K \leq \Gamma$ such that $g \not \in H_1H_2 K$. Define a graph
\[ X=\Gamma/H_1 \sqcup \Gamma/H_2 \]
with metric with values in the set $\{0,1,2\}$, and satisfying, for $fH_i \neq gH_j$, $d(fH_i,g H_j)=1$ iff $i=1$, $j=2$, and $f H_i \cap g H_j \neq \emptyset$. Clearly, $X$ is triangle-free, and $\Gamma$ acts on $X$ by left-translation. Note also that $d(H_1,gH_2)=2$. Let $\beta$ be an action on a finite graph $Y$ such that the left-translation action of $\Gamma$ on $X$ embeds into $\beta$ when restricted to $A=\{H_1,H_2,gH_2\}$ and $F=\{\mbox{generators of }H_1, H_2\} \cup \{1,g\}$. Let $e:X \rightarrow Y$ be such an embedding, and let $K_i$ be the stabilizer of $e(H_i)$, $i=1,2$. As $K_2$ has finite index in $\Gamma$, we can fix $K \leq K_2$ which has finite index and is normal in $\Gamma$.

Observe that $H_i \leq K_i$. Moreover, $g \not \in K_1K_2$ because for every $k_1 \in K_1$, $k_2 \in K_2$ we have that
\[ d(H_1,k_1k_2H_2)=d(H_1,H_2)=1. \] 
Therefore $g \not \in H_1H_2K$.
\end{proof}

\begin{rem}
We do not know whether the above bounds are strict, i.e., whether there exists a group with the $2$-RZ property, and with an action on a triangle-free graph that is not finitely approximable, or a group whose actions on triangle-free graphs are finitely approximable but which does not have the $3$-RZ property.
\end{rem}

\section{Finite approximability and generic representations}

In this section, we study connections between finite approximability and generic representations in the context of \fra theory. Moreover, we present a simpler proof of genericity of permutation representations from \cite{GKM}.

We say that a \fra class $\KK$ has \emph{independent amalgamation} if for any $A_0,A_1,A_2 \in \KK$ with $A_0 \leq A_1,A_2$ there exists an amalgam $B \in \KK$ of $A_1$ and $A_2$ over $A_0$ such that $A_1,A_2 \leq B$, and for all automorphisms $\phi_1$, $\phi_2$ of $A_1$, $A_2$, respectively, that agree on $A_0$, $\phi_1 \cup \phi_2$ extends to an automorphism of $B$. Analogously, we define \emph{independent joint embedding}. The simplest example of independent amalgamation is free amalgamation, present, e.g., in the class of finite graphs or $K_n$-free graphs.

Also, we say that a representation $\alpha$ of a group $\Gamma$ in the automorphism group $\Aut(M)$ of a structure $M$ has finite orbits if the orbit of every element of $M$ under the natural action of $\alpha[\Gamma]$ on $M$, is finite.

Recall that a chain from a \fra class $\KK$ is a locally finite structure $X$ such that $\age(X) \subseteq \KK$. It follows from \cite{KuMa} that if $\KK$ is a \fra class with a \kat functor, then for every chain $X$ in $\KK$ there exists an embedding of $X$ in the \fra limit $M$ of $\KK$ that gives rise to an embedding of $\Aut(X)$ in $\Aut(M)$. Thus, for \fra classes with a \kat functor (e.g., the class of finite graphs, by \cite[Example 2.5]{KuMa}, finite $K_n$-free graphs, by \cite[Example 2.10]{KuMa}, finite tournaments, by \cite[Example 2.6]{KuMa}, or finite metric spaces with rational distances, \cite[Example 2.4]{KuMa}), studying actions on chains (countable graphs, $K_n$-free graphs, tournaments or metric spaces with rational distances) reduces to studying actions on their \fra limits (the random graph, the random $K_n$-free graph, the random tournament or the rational Urysohn space.) 

\begin{thm}
\label{th:IndAmalg}
Let $\Gamma$ be a finitely generated group, and let $M$ be the \fra limit of a relational \fra class $\KK$ with independent amalgamation and joint embedding, and with a \kat functor. Then every action of $\Gamma$ on a chain from $\KK$ is finitely approximable if and only if $\Gamma$ has a generic representation in $\Aut(M)$ with finite orbits .
\end{thm}

\begin{proof}
If $\Gamma$ has a generic representation in $\Aut(M)$ with finite orbits, then, in particular, there exists a dense subset of $\Rep(\Gamma,M)$ consisting of representations with finite orbits. In other words, every action of $\Gamma$ on $M$ can be approximated by actions of $\Gamma$ on finite $A \in \KK$. Since $\KK$ has a \kat functor, this implies that every action of $\Gamma$ on a chain from $\KK$ is finitely approximable.
 
In order to prove the converse, we show that the class $\KK_a$ of all actions of $\Gamma$ on elements of $\KK$, with equivariant injections as morphisms, is a \fra class. The hereditary property is obvious. Amalgamation in $\KK_a$ follows from independent amalgamation in $\KK$: take actions $\alpha_0,\alpha_1, \alpha_2 \in \KK_a$ on $A_0,A_1,A_2 \in \KK$, respectively, such that $A_0 \subseteq A_1,A_2$, and  $\alpha_1 \rest \Gamma \times A_0 = \alpha_2 \rest \Gamma \times A_0=\alpha_0$. Fix an amalgam $B \in \KK$ of $A_1$ and $A_2$ over $A_0$ with $A_1 \cup A_2$ as the underlying set, and chosen using independent amalgamation in $\KK$. It is easy to verify that the mapping $\beta$ defined on $\Gamma \times B$ by setting $\beta(\gamma,x)=\alpha_1(\gamma,x)$ if $x \in A_1$, and $\beta(\gamma,x)=\alpha_2(\gamma,x)$ if $x  \in A_2 \setminus A_1$ is an action of $\Gamma$ on $B$. For the same reasons, $\KK_a$ has joint embedding. Because $\Gamma$ is finitely generated, $\KK_a$ is countable up to isomorphism, and so is a \fra class.

Now, $\KK$ has a \kat functor, which implies, together with the assumption that actions of $\Gamma$ on chains from $\KK$ are finitely approximable, that for every finite $A \subseteq B \subseteq M$, and every action $\alpha$ of $\Gamma$ on $A$, there exists a finite $C \subseteq M$ such that $B \subseteq C$, and an action $\beta$ on $C$ extending $\alpha$.
Thus, we can mimic the standard construction of a unique up to isomorphism \fra limit of $\KK_a$ (see \cite[Theorem 7.1.2]{Hod} for details): for any finite $A_0 \subseteq M$ and action $\alpha_0$ on $A_0$, we build an increasing sequence of finite $A_n \subseteq M$, $n>0 \in \NN$, such that $\bigcup_n A_n=M$, and an increasing sequence of actions $\alpha_n$ on $A_n$ so that the following condition is satisfied:

\medskip
for all $m \in \NN$, all actions $\beta, \beta'$ on finite subsets of $M$, and embeddings $e$, $f$ of $\beta$ in $\alpha_m$, $\beta'$, respectively, there is $n \geq m$, and an embedding $f'$ of $\beta'$ in $\alpha_n$ such that $f' \circ f=e$.
\medskip

Clearly, representations constructed as above form a dense set in $\Rep(\Gamma,M)$, and, for any such representation $\alpha$, the following condition is satisfied:

\medskip
for all actions $\beta, \beta'$ on finite subsets of $M$, and embeddings $e$, $f$ of $\beta$ in $\alpha$, $\beta'$, respectively, there is an embedding $f'$ of $\beta'$ in $\alpha$ such that $f' \circ f=e$.
\medskip

This means, that the set $C$ of all representations in $\Rep(\Gamma,M)$ that satisfy the above condition is a dense $G_\delta$, and, any two representations in $C$ are isomorphic (see \cite[Lemma 7.1.4]{Hod} for details), i.e., conjugate.

\end{proof}

\begin{cor}\label{cor_sect2_genericity}
Let $\Gamma$ be a finitely generated group.
\begin{enumerate}
\item If $\Gamma$ has the $3$-RZ property, then, for any $n \geq 3$, $\Gamma$ has a generic representation in $\Aut(\mathcal{G}_n)$, where $\mathcal{G}_n$ is the random $K_n$-free graph,
\item the group $\ZZ$ has a generic representation in $\Aut(\mathcal{T})$, where $\mathcal{T}$ is the random tournament; in particular, $\Aut(\mathcal{T})$ has a comeager conjugacy class,
\item more generally, if $\Gamma$ has the tournament $2$-RZ property, then $\Gamma$ has a generic representation in $\Aut(\mathcal{T})$,
\item (Rosendal) $\Gamma$ has the RZ property iff $\Gamma$ has a generic representation in $\Aut(\QU)$ with finite orbits.
\end{enumerate}
\end{cor}

\begin{proof}
As it has been already mentioned, the classes of finite $K_n$-free graphs, tournaments, and finite metric spaces with rational distances have \kat functors. The classes of $K_n$-free graphs have free amalgamation and joint embedding. For finite tournaments, the amalgamation is not free but this class has independent amalgamation (as well as joint embedding) because, for any given tournaments $A_0,A_1,A_2$ with $A_0 \subseteq A_0,A_1$, we can amalgamate $A_1$, $A_2$ over $A_0$ by adding an edge $(a_1,a_2)$ iff $a_1 \in A_1$, $a_2 \in A_2$, for $a_1,a_2 \not \in A_0$. Similarly, the class of finite metric spaces with rational distances has metric-free amalgamation, which is also a case of independent amalgamation.

Then we use Theorem \ref{th:3RZKnFree} for $K_n$-free graphs, Proposition \ref{pr:Ztour2RZ} and Theorem \ref{thm_strong2RZ} for tournaments, and \cite[Theorem 7]{Ro} for metric spaces.

\end{proof}
\begin{rem}
We note that Corollary \ref{cor_sect2_genericity} in particular implies that $K_n$-free graphs, for $n\geq 3$, have the Hrushovski property and the automorphism group of the random $K_n$-free graph has ample generics. This is a result originally proved by Herwig in \cite{He}.
\end{rem}


\subsection{The theorem of Glasner, Kitroser and Melleray}
Finally, we give another proof of a characterization of groups with generic permutation representations that was proved by Glasner, Kitroser and Melleray in \cite{GKM}, which is in the spirit of our other proofs in this paper. Let us recall some terminology. For a countable group $\Gamma$, view the set $\Sub(\Gamma)$ of its subgroups as a closed subspace of the Cantor space $2^\Gamma$. With this identification we give $\Sub(\Gamma)$ a compact Polish topology usually called the Chabauty topology (see \cite{Cha} for the original reference). Glasner, Kitroser and Melleray call a countable group $\Gamma$ \emph{solitary} if the set of isolated points in $\Sub(\Gamma)$ is dense.

\begin{thm}[Glasner, Kitroser, Melleray]
\label{th:GlKiMe}
A countable group $\Gamma$ has a generic permutation representation if and only if it is solitary.
\end{thm}
\begin{rem}
Notice the analogy of this theorem with the result that follows from \cite{KLP} mentioned in the beginning of Section \ref{section4} which says that $\Gamma$ has a generic unitary representation if and only if the set of isolated points in the unitary dual $\hat{\Gamma}$ equipped with the Fell topology is dense.
\end{rem}
\begin{proof}
Let $\Gamma$ be a countable solitary group. Let $\mathcal{G}=\{\Gamma_i: i\in\Nat\}$ be the (at most) countable collection of isolated subgroups of $\Gamma$ in $\Sub(\Gamma)$ which form a dense subset there. Let us identify $\NN$ with the disjoint union of countably many copies of cosets of subgroups from $\mathcal{G}$, and let $\Age$ be the set of all finite `sums' of left quasi-regular actions
\[ \Gamma\curvearrowright \Gamma/H_1 \sqcup \ldots \sqcup \Gamma/H_n, \]
where $H_i\in \mathcal{G}$, for $i\leq n$. It is immediate that it is a countable class with the amalgamation property. Thus, we can construct its unique up to isomorphism \fra limit $\alpha \in \Rep(\Gamma,S_\infty)$ as in the proof of Theorem \ref{th:IndAmalg}. In this way, we obtain $C \subseteq \Rep(\Gamma,S_\infty)$ consisting of isomorphic representations $\alpha$ that satisfy the following condition: 

\medskip
for all $H \in \mathcal{G}$, $m \in \NN$, and $n_1, \ldots, n_m \in \NN$,  there is  $n\in \Nat\setminus \alpha[\Gamma][\{n_1,\ldots, n_m\}]$ such that $\Stab^\alpha_n=H$.
\medskip

Clearly, $C$ dense. In order to see that it is $G_\delta$, we need to verify that `$\Stab^\beta_x=H$' is an open condition, which, however, follows since $H$ is isolated, thus uniquely determined by finitely many group elements.\\

Suppose that $\Gamma$ is not solitary. Then there exists a non-empty basic open set $O$ without isolated points in $\Sub(\Gamma)$, consisting of subgroups of $\Gamma$ containing some $g_1,\ldots,g_n\in \Gamma$, and not containing some $h_1,\ldots,h_m\in \Gamma$. For every $H\in O$, let 
\[ C(H)=\{\alpha\in\Rep(\Gamma,S_\infty):\forall n\in\Nat \, \exists g\in \Gamma (\alpha(g).n=n\Leftrightarrow g\notin H)\}. \]
It is easy to check it is a $G_\delta$ set. Moreover it is dense. Indeed, fix a basic open neighborhood $U$ of some $\beta\in\Rep(\Gamma,S_\infty)$, and $n\in\Nat$. It suffices to show that there is $\gamma\in U$ such that for some $g\in \Gamma$, 
\[ \gamma(g).n=n\Leftrightarrow g\notin H. \]
If $\beta$ satisfies this condition, we are done. Otherwise, fix $n_1,\ldots,n_k\in\Nat$ and $g_1,\ldots,g_l\in \Gamma$ that determine $U$; we can suppose $n_1=n$. Since $H$ is not isolated, there exist $H'\leq \Gamma$ and $h\in \Gamma\setminus F$ such that for $l'\leq l$, $g_k\in H$ iff $g_k\in H'$ and $h\in H$ iff $h\notin H'$. Therefore we can find some $\gamma\in U$ such that $\gamma(\Gamma).n\cong \Gamma/H'$, which means that $\gamma$ satisfies the condition above.

Now suppose there is a comeager conjugacy class $C$. It must intersect the open set
\[ D=\{\alpha\in\Rep(\Gamma,S_\infty):\exists n\in\Nat\;\exists H\in O\; \alpha(\Gamma).n\cong \Gamma/H\}.\]
Then for some $H\in O$, and for every $\beta\in C$, there is $n\in\Nat$ such that $\beta(\Gamma).n\cong \Gamma/H$. This contradicts that $C$ also intersects $C(H)$.
\end{proof}

\section{Generic representations in metric structures}\label{section4}
In this section, we investigate generic properties of representations of countable groups in automorphism groups of metric structures. Typically, there are no generic representations in this situation. However, perhaps the most interesting case, when the metric structure in question is the separable infinite-dimensional Hilbert space, is still open. It follows from Theorem 2.5 in \cite{KLP} that when $\Gamma$ is a group with the Kazhdan's property (T) whose finite-dimensional unitary representations form a dense set in the unitary dual $\hat{\Gamma}$, $\Gamma$ has a generic unitary representation. See also \cite{DMV} for a more explicit statement of this theorem and more elementary proof. Nevertheless, the existence of such infinite Kazhdan groups is an open question, see e.g. Question 7.10 in \cite{BdlHV}.

Here, we prove that an at most countable group $\Gamma$ has a generic representation in the isometry group of the Urysohn space and the Urysohn sphere if and only if $\Gamma$ is finite. As mentioned in the introduction, Julien Melleray informed us he had proved it earlier in his habilitation thesis, see Theorem 5.78 in \cite{Me}. As he did not publish the proof for the Urysohn sphere we use the opportunity to present our proof here (for the sake of completeness also with the proof for the Urysohn space which is a simpler version). We also show that every infinite countable group has meager conjugacy classes in the linear isometry group of the Gurarij space, which answers a question of Melleray from the same paper \cite{Me}. We also show that these methods can be used to prove that when one restricts to the space of free actions on the Random graph, the rational Urysohn space, etc., then every infinite group has meager classes.

Most importantly, we show that the conjugation action of the isometry group of the Urysohn space on the space of representations of a fixed infinite group $\Gamma$ in the Urysohn space is generically turbulent. We recall from the introduction that Kechris and Rosendal asked in \cite{KeRo} if that is true for $\Gamma=\Int$.

\subsection{Urysohn space and Urysohn sphere} Denote by $\Ur$ and by $\Ur_1$ the Urysohn universal metric space and the Urysohn sphere respectively. We refer the reader to Chapter 5 in \cite{Pe} for information about the Urysohn space. We recall that the \fra limit of finite rational-valued metric spaces is the so called rational Urysohn space, denoted here by $\Rat\Ur$, and $\Ur$ is its completion. Analogously, the \fra limit of finite rational-valued metric spaces bounded by one is the rational Urysohn sphere, denoted by $\Rat\Ur_1$, and its completion is $\Ur_1$. Alternatively, one may obtain $\Ur_1$, as the name suggests, by picking any point from $\Ur$ and taking the subset of $\Ur$ of those points having distance one half from the chosen point.

For a fixed countable group $\Gamma$, we shall denote the Polish space of representations of $\Gamma$ in the Polish groups $\Iso(\Ur)$, resp. $\Iso(\Ur_1)$ by $\Rep(\Gamma,\Ur)$, resp. $\Rep(\Gamma,\Ur_1)$.

\begin{thm}
Let $\Gamma$ be a finite group. Then the classes of all actions of $\Gamma$ on finite rational metric spaces, resp. on finite rational metric spaces bounded by $1$ are Fra\" iss\' e classes whose limits are an action of $\Gamma$ on the $\Rat\Ur$, resp. on $\Rat\UU_1$. Their completions are generic actions of $\Gamma$ on $\Ur$, resp. on $\UU_1$.
\end{thm}
We omit the proof which is straightforward, and probably known among the experts. Let us only note that the extension of the generic action of $\Gamma$ on $\Rat\UU$ (resp. on $\Rat\UU_1$) to the completion $\UU$ (resp. $\UU_1$) can still be uniquely described by the finite extension property, up to $\varepsilon$, which provides a $G_\delta$ condition for the equivalence class of this action inside $\Rep(\Gamma,\UU)$ (resp. $\Rep(\Gamma,\UU_1)$). The density is clear, so we get that the equivalence class of the completion of the \fra limit is comeager.
\medskip

When one considers the class of all actions of an infinite countable group $\Gamma$ on finite rational metric spaces, it is still a \fra class - although it can be empty if $\Gamma$ does not have finite quotients. The equivalence class of the completion of the \fra limit can even happen, in some cases, to be dense in $\Rep(\Gamma,\UU)$, however the computation that it is $G_\delta$ fails.\\

Recall that a \emph{pseudonorm} (or \emph{length function}) on a group $\Gamma$ is a function $\lambda:\Gamma\rightarrow \Rea$ satisfying $\lambda(1_\Gamma)=0$, $\lambda(g)=\lambda(g^{-1})$, for $g\in \Gamma$, and $\lambda(g\cdot h)\leq \lambda(g)+\lambda(h)$, for $g,h\in\Gamma$. We shall generalize this notion below.
\begin{defin}
Let $\Gamma$ be a group and $I$ some index set. A \emph{generalized pseudonorm} on the pair $(\Gamma,I)$ is a function $N:\Gamma\times I^2\rightarrow [0,\infty)$ satisfying
\begin{itemize}
\item $N(g,i,j)=N(g^{-1},j,i)$, for all $g\in \Gamma$ and $i,j\in I$;
\item $N(1_\Gamma,i,i)=0$, for every $i\in I$, and $N(g,i,j)>0$ for all $g\in \Gamma$ (including $1_\Gamma$) whenever $i\neq j$;
\item $N(g h,i,j)\leq N(g,i,k)+N(h,k,j)$, for all $g,h\in \Gamma$ and $i,j,k\in I$.

\end{itemize}
For any $i\in I$, we shall denote by $N_i$ the function defined by $N_i(g)=N(g,i,i)$. Clearly, it is a pseudonorm on $\Gamma$.
\end{defin}

Generalized pseudonorms correspond to actions of $\Gamma$ on metric spaces by isometries together with representatives of each orbit. Indeed, let $I$ be some index set, and let $N$ be a generalized pseudonorm on $(\Gamma,I)$. For each $i\in I$, let $H_i\leq \Gamma$ be the subgroup defined as the kernel of $N_i$. We define a metric $d$ on $M=\bigcup_{i\in I} \Gamma/H_i$ as follows: for $g,h\in \Gamma$ and $i,j\in I$ we set $d(g H_i, h H_j)=N(g^{-1} h,i,j)$. It is straightforward to check that it is a metric, and, moreover, that the natural action of $\Gamma$ on $M$ (defined as $g.(h H_i)=(g h) H_i$) is an action by isometries.

Conversely, let $(M,d)$ be a metric space on which $\Gamma$ acts by isometries. Let $I$ be an index set for all the orbits in $M$ of this action, and for each $i\in I$ select a representative $x_i\in M$ from this orbit. Now the function $N:\Gamma\times I^2\rightarrow \Rea$ defined by $N(g,i,j)=d(x_i,g x_j)$ is readily checked to be a generalized pseudonorm.\\

A function $P:\Gamma\times I^2\rightarrow [0,\infty)$ satisfying all the axioms of the generalized pseudonorm except the triangle inequality is called \emph{generalized pre-pseudonorm}.
\begin{fact}\label{fact_extension}
If $P:\Gamma\times I^2\rightarrow \Rea$ is a generalized pre-pseudonorm, then there exists a maximal generalized pseudonorm $N$ satisfying $N(g,i,j)\leq P(g,i,j)$ for all $g\in \Gamma$ and $i,j\in I$.
\end{fact}
Maximality of $N$ means that for any generalized pseudonorm $N'$ satisfying $N'(g,i,j)\leq P(g,i,j)$, for all $g\in\Gamma$ and $i,j\in I$, we have $N'(g,i,j)\leq N(g,i,j)$ for all such triples.
\begin{proof}
Consider a complete graph with the set of vertices $\Gamma\times I$, i.e. $|I|$ disjoint copies of $\Gamma$ and consider a real valuation of its edges $P'$, where the value $P'((g,i),(h,j))$ is defined to be $P(h^{-1} g,i,j)$, for $g,h\in \Gamma$ and $i,j\in I$. Now consider the corresponding graph metric $N'$ on $\Gamma\times I$. It is clear that the function $N(g,i,j)=N'((1_\Gamma,i),(g,j))$ is the desired generalized pseudonorm. Note that another equivalent way how to define $N$ is to set
\[ N(g,i,j)=\inf\{\sum_{i=1}^m P(g_i,k_i,l_i):g=g_1\ldots g_m,l_i=k_{i+1}\forall i<m\}, \]
from which the maximality of $N$ is easily seen.
\end{proof}
An analogous notion is that of a partial generalized pseudonorm, which is a function satisfying all the axioms of the generalized pseudonorm except that it is defined partially, i.e. its domain is a proper subset of $\Gamma\times I^2$. We provide a precise definition below.
\begin{defin}
A function $P:A\subseteq\Gamma\times I^2\rightarrow \Rea$ is a \emph{partial generalized pseudonorm} if
\begin{itemize}
    \item whenever $(1_\Gamma,i,i)\in A$, for some $i\in I$, then $P(1_\Gamma,i,i)=0$;
    \item whenever $(g,i,j)\in A$, for some $g\in\Gamma$ and $i\neq j\in I$, then $P(g,i,j)>0$;
    \item whenever $(g,i,j),(g^{-1},j,i)\in A$, for some $g,i,j$, then $P(g,i,j)=P(g^{-1},j,i)$;
    \item whenever $(g_1\ldots g_n,i_0,i_n),(g_1,i_0,i_1),\ldots,(g_n,i_{n-1},i_n)\in A$, for some $g_1,\ldots,g_n\in\Gamma$ and $i_0,\ldots,i_n\in I$, then $P(g_1\ldots g_n,i_0,i_n)\leq P(g_1,i_0,i_1)+\ldots+P(g_n,i_{n-1},i_n).$
\end{itemize}
\end{defin}
\begin{rem}
We remark that in the previous definition, it is indeed necessary to require $P(g_1\ldots g_n,i_0,i_n)\leq P(g_1,i_0,i_1)+\ldots+P(g_n,i_{n-1},i_n)$, where $n$ can be arbitrarily large. While for the (full) generalized pseudonorms, the general case follows by induction from the case $n=2$, i.e. from the standard triangle inequality, this is not necessarily anymore true in the partial case. Consider e.g. the case when $I$ is a singleton (so we omit it from the notation), $\Gamma=\Int$ and $A=\{1,3\}\subseteq \Int$. Then we may set $P(1)=1$ and $P(3)=4$. We do not violate the standard triangle inequality since $2\notin A$, however we do not have $P(3)\leq P(1)+P(1)+P(1)$.
\end{rem}
 Let now $P:A\subseteq \Gamma\times I^2\rightarrow \Rea$ be some partial generalized pseudonorm. Take again a graph $V$ with the set of vertices $\Gamma\times I$ and connect two vertices $(g,i)$ and $(h,j)$ by an edge if and only if $(h^{-1}g,i,j)\in A$. If $V$ is connected, then we say that the partial generalized pseudonorm is \emph{sufficient}. Note that this is the case when for example $A=F\times I^2$, where $F$ is some symmetric generating subset of $\Gamma$. Also, in this case we can extend the partial generalized pseudonorm onto a genuine generalized pseudonorm.
\begin{fact}\label{fact_extension2}
Let $P:A\subseteq \Gamma\times I^2\rightarrow \Rea$ be a sufficient partial generalized pseudonorm. Then there exists a maximal generalized pseudonorm $N$ on $\Gamma\times I^2$ which extends $P$.
\end{fact}
\begin{proof}
The proof proceeds as the proof of Fact \ref{fact_extension}. We take a graph $V$ with the set of vertices $\Gamma\times I$ and we connect two vertices $(g,i)$ and $(h,j)$ by an edge if and only if $(h^{-1}g,i,j)\in A$, as above. For every edge $((g,i),(h,j))$ in $V$ we define a value $P'((g,i),(h,j))$ of this edge to be $P(h^{-1} g,i,j)$. Then we take the graph metric which gives the generalized pseudonorm $N$ as in the proof of Fact \ref{fact_extension}. This is again equivalent to defining $N$ as follows, for $g\in \Gamma$ and $i,j\in I$ we set

$$N(g,i,j):=\inf\{P(g_1,i_0,i_1)+\ldots+P(g_n,i_{n-1},i_n):$$ $$g=g_1\ldots g_n, i_0,\ldots,i_n\in I, i_0=i,i_n=j,(g_1,i_0,i_1),\ldots,(g_n,i_{n-1},i_n)\in A \}.$$

From this definition and properties of $P$, it easily follows that $N$ extends $P$.
\end{proof}
We are now prepared to prove the meagerness of conjugacy classes for the Urysohn space and the Urysohn sphere.
\begin{thm}[see also Theorem 5.78 in \cite{Me}]\label{thm_meagerclassesUrysohn}
Let $\Gamma$ be an infinite group. Then every $\alpha \in \Rep(\Gamma,\Ur)$ has all conjugacy classes meager. Analogously, every $\alpha \in \Rep(\Gamma,\UU_1)$ has all conjugacy classes meager.
\end{thm}

\begin{proof}
Fix an infinite group $\Gamma$. First, we work with the full Urysohn space. Let $\lambda$ be an arbitrary pseudonorm on $\Gamma$. It suffices to show that the set $C(\lambda)$ of all $\alpha \in \Rep(\Gamma,\Ur)$ such that for every $x \in \Rat\Ur$ there exists $g \in \Gamma$ with 
\[ \left|\lambda(g)- d_\Ur(\alpha(g)x,x)\right|>1/4 \]
is comeager. Indeed, if the conjugacy class of any $\alpha'\in\Rep(\Gamma,\Ur)$ was non-meager, it would have a non-empty intersection with $C(\lambda_0)$, where $\lambda_0(g)=d_\Ur(\alpha'(g)x,x)$ for some fixed $x \in \Rat\Ur$; this is clearly a contradiction.

To see that $C(\lambda)$, for any pseudonorm $\lambda$ on $\Gamma$, is comeager, it suffices to show that for a fixed $x \in \Rat\Ur$ the open set of all $\alpha \in \Rep(\Gamma,\Ur)$ such that there exists $g \in \Gamma$ with
\[ \left|\lambda(g)- d_\Ur(\alpha(g)x,x)\right|>1/4, \]
is dense in $\Rep(\Gamma,\Ur)$.

Let $U$ be an open neighborhood of some $\beta\in\Rep(\Gamma,\Ur)$ given by some finite set $x_1,\ldots,x_n$, where $x=x_1$ and each $x_i,x_j$ lie in different orbits of $\beta$, some finite symmetric set $F\subseteq \Gamma$ containing the unit and some $\varepsilon>0$. Set $F'=\{g^{-1} h:g,h\in F\}$, $I=\{1,\ldots,n\}$ and let $N_0:F'\times I^2\rightarrow \Rea$ be a function defined as $N_0(g,i,j)=d_\Ur(x_i,\beta(g)x_j)$, for every $g\in F'$, $i,j\in I$. Let $N'_0:F'\times I^2\rightarrow \Rea$ be a function defined as $$N'_0(g,i,j)=\begin{cases} 0 & g=1_\Gamma, i=j;\\
N(g,i,j)+\varepsilon/2 & \text{otherwise}.
\end{cases}$$
Note that $N'_0$ is a partial generalized pseudonorm. Let
\[ M'=\max\{N'_0(g,i,j):g\in F',i,j\in I\}.\] 
If the pseudonorm $\lambda$ is bounded by some $K$, then we set $M=\max\{M',K+1/4\}$; if it is unbounded, we set $M=M'$. Finally, let $N''_0:\Gamma\times I^2\rightarrow \Rea$ be a generalized pre-pseudonorm extending $N'_0$ which coincides with $N'_0$ on its domain, and everywhere else it is constantly $M$. Now we use Fact \ref{fact_extension} to find a maximal generalized pseudonorm $N$ bounded by $N''_0$.

Note that 
\begin{itemize}
\item $N(g,i,j)=0$ if and only if $g=1_\Gamma$ and $i=j$;
\item $N(g,i,j)=N'_0(g,i,j)$ if $g\in F'$;
\item for every $i\in I$ for all but finitely many $g\in \Gamma$ we have $N(g,i,i)=M$.

\end{itemize}
$N$ gives an action $\gamma'$ by isometries on a metric space $Y=\bigcup_{i\leq n} \Gamma. y_i$, where $d_Y(\gamma'(g) y_i,\gamma'(h) y_j)=N(h^{-1} g,i,j)$, for $g,h\in \Gamma$ and $i,j\in I$. Using the Kat\v etov functor for metric spaces we can extend the action $\gamma'$ on $Y$ to an action $\gamma$ on $\Ur$, and then, using the finite extension property of the Urysohn space, we can find an isometry $\phi$ of $\Ur$ such that $\phi(y_1)=x_1$, and for every $i\leq n$ and $g\in F$ we have $d_\Ur(\phi(\gamma(g)y_i),\beta(g)x_i)<\varepsilon$. Without loss of generality, we can assume that $\phi$ is the identity map, therefore $\gamma$ is in the neighborhood $U$ of $\beta$. By the construction, there exists $g\in \Gamma$ such that $d_\Ur(x,\gamma(g)x)=M$, while $\lambda(g)\leq M-1/4$ in case $\lambda$ was bounded, or $\lambda(g)>M+1/4$ in case $\lambda$ was unbounded. This finishes the proof for the Urysohn space.\\

Now we work with the Urysohn sphere. Fix some pseudonorm $\lambda$ on $\Gamma$ which is now bounded by $1$. Analogously as above, it suffices to show that for a fixed $x \in \Rat\UU_1$ the open set of all $\alpha \in \Rep(\Gamma,\Ur)$ such that there exists $g \in \Gamma$ with
\[ \left|\lambda(g)- d_{\UU_1}(\alpha(g)x,x)\right|\geq 1/4, \]
is dense in $\Rep(\Gamma,\UU_1)$.

The proof proceeds similarly in the beginning, however we cannot take an advantage of the fact that $\lambda$ may be unbounded. Let $U$ be an open neighborhood of some $\beta\in\Rep(\Gamma,\UU_1)$ again given by some finite set $x_1,\ldots,x_n$, where $x=x_1$, and each $x_i,x_j$ lies in a different orbit of $\beta$, some finite symmetric set $F\subseteq \Gamma$ containing the unit, and some $\varepsilon>0$. Set $F'=\{g^{-1} h:g,h\in F\}$, $I=\{1,\ldots,n\}$, and let $N_0:F'\times I^2\rightarrow \Rea$ be a function defined as $N_0(g,i,j)=d_{\UU_1}(x_i,\beta(g)x_j)$, for every $g\in F'$, $i,j\in I$. Let $N'_0:F'\times I^2\rightarrow \Rea$ be a function defined as $$N'_0(g,i,j)=\begin{cases} 0 & g=1_\Gamma, i=j;\\
\max\{N(g,i,j)+\varepsilon/2,1\} & \text{otherwise}.
\end{cases}$$
Set $m=\min\{N'_0(g,i,j):g\neq 1_\Gamma\text{ or }i\neq j\}$. Note that $m>0$. Take now $M=\lceil 1/m\rceil+1$.

Suppose first that there are infinitely many $g\in \Gamma$ such that $\lambda(g)\leq 3/4$. Then we set $N''_0:\Gamma\times I^2\rightarrow \Rea$ to be a generalized pre-pseudonorm extending $N'_0$ which coincides with $N'_0$ on its domain, and everywhere else it is constantly $1$. Now we use Fact \ref{fact_extension} to find a maximal generalized pseudonorm $N$ bounded by $N''_0$. It is clear that $N$ coincides with $N'_0$ on its domain and it is equal to $1$ at all but finitely many elements.

If, on the other hand, for all but finitely many $g$ we have $\lambda(g)>3/4$, then we may and will find some $g\in \Gamma$ satisfying $\lambda(g)>3/4$ and such that $g$ cannot be obtained as a product of less than $M$-many elements of $F'$. Then we set $N''_0:\Gamma\times I^2\rightarrow \Rea$ to be a generalized pre-pseudonorm extending $N'_0$  which coincides with $N'_0$ on its domain, it is equal to $1/2$ for at $(g,1,1)$, and everywhere else it is constantly $1$. We again use Fact \ref{fact_extension} to find a maximal generalized pseudonorm $N$ bounded by $N''_0$. Clearly $N(g,1,1)\leq 1/2$. We now check that $N$ coincides with $N'_0$ on its domain. Suppose that there is some $(f,i,j)$ from the domain of $N'_0$ such that $N(f,i,j)<N'_0(f,i,j)$. By the construction of $N$, it means there are $g_1,\ldots,g_n\in F'$ and $k_1,l_1,\ldots,k_n,l_n\in I$ such that $f=g_1\ldots g_n$ and $l_i=k_{i+1}$ for $i<n$ and $N(f,i,j)=\sum_{i=1}^n N''_0(g_i,k_i,l_i)<N'_0(f,i,j)$. We claim that $n<M-1$ since otherwise $N(f,i,j)=\sum_{i=1}^n N''_0(g_i,k_i,l_i)\geq \sum_{i=1}^n m\geq 1$ (note that for each $i\leq n$ we have $N''_0(g_i,k_i,l_i)\geq m$).  It is clear that for at least one $i\leq n$ we must have $(g_i,k_i,l_i)=(g,1,1)$ since otherwise all $(g_i,k_i,l_i)$ are from the domain of $N'_0$, and $N'_0$ is a restriction of a genuine generalized pseudonorm, so the triangle inequalities are satisfied there. On the other hand, there is at most one $i\leq n$ such that $(g_i,k_i,l_i)=(g,1,1)$ since $N''_0(g,1,1)=1/2$. Let $i\leq n$ be such that $g_i=g$. We have $f=g_1\ldots g_{i-1} g g_{i+1}\ldots g_n$, so $g=g^{-1}_{i-1}\ldots g^{-1}_1 f g^{-1}_n\ldots g^{-1}_{i+1}$, but that is in contradiction with the assumption that $g$ cannot be obtain as a product of less than $M$-many elements of $F'$.

Now we finish the proof as in the Urysohn space case. $N$ gives an action $\gamma'$ by isometries on a metric space $Y=\bigcup_{i\leq n} \Gamma.y_i$ bounded by $1$, where $d_Y(\gamma'(g) y_i,\gamma'(h) y_j)=N(h^{-1} g,i,j)$, for $g,h\in \Gamma$ and $i,j\in I$. Using the Kat\v etov functor for metric spaces bounded by $1$, we can extend the action $\gamma'$ on $Y$ to an action $\gamma$ on $\UU_1$, and then, using the finite extension property of the Urysohn sphere, we can find an isometry $\phi$ of $\Ur_1$ such that $\phi(y_1)=x_1$, and for every $i\leq n$ and $g\in F$ we have $d_{\UU_1}(\phi(\gamma(g)y_i),\beta(g)x_i)<\varepsilon$. Without loss of generality, we can assume that $\phi$ is the identity map, therefore $\gamma$ is in the neighborhood $U$ of $\beta$. 

Now if $\lambda$ was such that for infinitely many $g$ we have $\lambda(g)\leq 3/4$, then we have guaranteed that there are co-finitely many $g\in \Gamma$ such that $d_{\Ur_1}(x,\gamma(g)x)=1$. If, on the other hand, for all but finitely many $g$ we have $\lambda(g)>3/4$, then we have guaranteed an existence of $g\in \Gamma$ such that $\lambda(g)>3/4$, while $d_{\Ur_1}(x,\gamma(g)x)\leq 1/2$. This finishes the proof. 
\end{proof}
An object analogous to the Urysohn space in the category of Banach spaces is the Gurarij space (\cite{Gu}), denoted here by $\Gu$. We refer the reader to the paper \cite{KuSo} for information about this Banach space. We can use similar methods as in the proof of Theorem \ref{thm_meagerclassesUrysohn} to show that every representation of every infinite group has meager conjugacy classes in the linear isometry group of the Gurarij space. Since the arguments are repetitive we provide only a sketch of the proof. The existence of the Kat\v etov functor in the category of Banach spaces is shown in \cite{BY}. We mention that this theorem answers Question 5.7 from \cite{Me}.
\begin{thm}\label{thm_meagerclassesGurarij}
Let $\Gamma$ be a countably infinite group. Then every conjugacy class in $\Rep(\Gamma,\Gu)$ is meager.
\end{thm}
\begin{proof}[Sketch of the proof.]
Fix an infinite group $\Gamma$ and some countable dense subset $D$ of the sphere in $\Gu$. Notice that for any $\alpha\in\Rep(\Gamma,\Gu)$ and any $x\in \Gu$ of norm one, the function on $\Gamma$ defined as $g\to \|\alpha(g)x-x\|$ is again a pseudonorm (bounded by $2$).  Therefore, as in the case of the Urysohn sphere, it suffices to show that for any pseudonorm $\lambda$ on $\Gamma$, bounded by $2$, we have that the set of those $\alpha \in \Rep(\Gamma,\Gu)$ such that for every $x \in D$ there exists $g \in \Gamma$ with 
\[ \left|\lambda(g)- \|\alpha(g)x-x\|\right|>1/4 \]
is comeager. That again reduces to showing that for a fixed $x \in D$ the open set of all $\alpha \in \Rep(\Gamma,\Gu)$ such that there exists $g \in \Gamma$ with
\[ \left|\lambda(g)- \|\alpha(g)x-x\|\right|>1/4, \]
is dense in $\Rep(\Gamma,\Gu)$.

Take $U$ to be an open neighborhood of some $\beta\in\Rep(\Gamma,\Gu)$ given by some finite set $x_1,\ldots,x_n\in D$ of unit linearly independent vectors, where $x=x_1$ and each $x_i,x_j$ lie in different orbits of $\beta$, some finite symmetric set $F\subseteq \Gamma$ containing the unit and some $\varepsilon>0$. Perturbing the representation $\beta$ by less than $\varepsilon$ if necessary, we may assume that the set $S=\{\beta(f)x_i: f\in F,i\leq n\}$ consists of linearly independent elements. As in the Urysohn sphere case, we get that $N':F\times I^2\rightarrow\Rea$ defined as $N'(f,i,j)=\|x_i-\beta(f)x_j\|$ is a partial generalized pseudonorm, where $I=\{1,\ldots,n\}$. By the same arguments as for the Urysohn sphere we get that $N'$ extends to a generalized pseudonorm $N:\Gamma\times I^2\rightarrow\Rea$ that `avoids' $\lambda$, i.e. there exists $g\in\Gamma$ such that $|\lambda(g)-N(g,1,1)|>1/4$.

Now let $X$ be the vector space spanned by the set $\{g.x_i: g\in\Gamma,i\leq n\}$. Note that $\Gamma$ acts canonically on $X$, so for any $g\in\Gamma$ and $x\in X$ the element $g.x$ is defined. By $Y$ we denote the finite-dimensional subspace spanned by $\{g.x_i:g\in F,i\leq n\}$. We define a norm on $X$. First we define `a partial norm' $\|\cdot\|'$ on $X$ and then show how it naturally extends to a norm on $X$ so that the canonical action of $\Gamma$ on $X$ is by linear isometries. Since we may identify $Y$ with the finite-dimensional subspace of $\Gu$ spanned by $S$ we set for any $x\in Y$, $$\|x\|'=\|x\|_\Gu.$$ Next, for any $g,h\in\Gamma$ and $i,j\leq n$ we set $$\|g.x_i-h.x_j\|'=\|h.x_j-g.x_i\|'=N(g^{-1}h,i,j).$$ Note that in case that $g.x_i-h.x_j\in Y$ we have $\|g.x_i-h.x_j\|_\Gu=N(g^{-1}h,i,j)$, so our definition is consistent. For any $g\in\Gamma$ and $i\leq n$ we set $$\|g.x_i\|'=1,$$ this is again consistent, and finally we make $\|\cdot\|'$ invariant under the canonical action of $\Gamma$, i.e. for any $g\in\Gamma$ and $x\in X$ such that $\|x\|'$ has been defined above we set $$\|g.x\|'=\|x\|'.$$ This is again readily checked to be consistent.

Now we set $\|\cdot\|$ to be the greatest norm on $X$ that extends $\|\cdot\|'$. That can be formally defined as follows. For any $x\in X$ we set $$\|x\|=\inf\{\sum_{i=1}^m |\alpha_i|\|x_i\|':\; (\alpha_i)_{i\leq m}\subseteq \Rea,\; (x_i)_{i\leq m}\subseteq \mathrm{dom}(\|\cdot\|'),\; x=\sum_{i=1}^m \alpha_i x_i\}.$$

Finally, using the Kat\v etov functor for Banach spaces, we may extend the action $\Gamma$ on $X$ to an action $\gamma$ of $\Gamma$ on $\Gu$, and moreover in such a way that $\gamma\in U$ and $|\|\gamma(g)x-x\|-\lambda(g)|>1/4$ which is what we were supposed to show.
\end{proof}
\begin{rem}
One may ask what happens when $\Gamma$ is a finite group. Is there a generic representation of $\Gamma$ in $\Gu$? One can check Section 4 in \cite{Do} where \fra classes of representations of groups in Banach spaces were considered. Although we have not formally verified it, we believe the class of finite-dimensional Banach spaces with group actions considered there, if the finite group $F$ is fixed, is a \fra class and the corresponding limit is a generic representation of $F$ in $\Gu$.
\end{rem}

\subsection{Free actions on countable metric spaces}
Let $\Gamma$ be a countable discrete group and $X$ a countable structure. Notice that the set $\Rep_F(\Gamma,X)$ of all free actions of $\Gamma$ on $X$ is a $G_\delta$ set invariant under the conjugacy action of $\Gamma$ and therefore is a Polish space itself. One may thus also study conjugacy classes in these spaces. It is obvious from the proofs of Theorems \ref{thm_meagerclassesUrysohn} and \ref{thm_meagerclassesGurarij} that the main difference between spaces $\Rep(\Gamma,\Aut(M))$ and $\Rep(\Gamma,\Iso(N))$, where $M$ is a countable metric space viewed as a countable discrete structure and $N$ is a Polish metric space, that in the latter `locally free' actions are dense. Here by `locally free actions are dense' we mean that for any finite set $\{x_1,\ldots,x_n\}$ the set of those actions whose restriction on the orbit of $x_i$, for $i\leq n$, is regular is dense. Since this is also satisfied, for obvious reasons, in the spaces of free actions we can prove by the same means as in the proof of Theorem \ref{thm_meagerclassesUrysohn} the following.
\begin{thm}
Let $\Gamma$ be a countably infinite group. Then all conjugacy classes are meager in the spaces $\Rep_F(\Gamma,\Aut(\QU))$, $\Rep_F(\Gamma,\Aut(\QU_1))$ and $\Rep_F(\Gamma,\Aut(\mathcal{R}))$, where $\mathcal{R}$ is the random graph.
\end{thm}

This stands in stark contrast with Rosendal's results from \cite{Ro2} which say that all finitely generated groups with the RZ property have a generic representation in $\Iso(\QU)$, and all finitely generated groups with the 2-RZ property have a generic representation in $\Aut(\mathcal{R})$.

\subsection{Generic turbulence}
The notion of turbulence was introduced by Hjorth in \cite{Hjo} in order to develop methods for proving non-classifiability by countable structures. Suppose that $G$ is a Polish group acting continuously on a Polish space $X$. Fix a point $x\in X$, some open neighborhood $U$ of $x$ in $X$ and some open neighborhood $V$ of the unit in $G$. The \emph{local orbit} $O(U,V)$ of $x$ is the set $\{y\in U:\exists g_1,\ldots, g_n\in V\; (y=g_1\ldots g_n.x\wedge \forall i\leq n\; (g_1\ldots g_i.x\in U))\}$.

A point $x\in X$ is \emph{turbulent} if for every open neighborhood $U$ of $x$ in $X$ and every open neighborhood $V$ of $1_G\in G$, the local orbit $O(U,V)$ of $x$ is somewhere dense (in $U$). If there is a $G$-invariant comeager subset $Y\subseteq X$ such that every $G$-orbit in $Y$ is dense and meager and every point in $Y$ is turbulent, then we say that the $G$-action on $X$ is \emph{generically turbulent}. It is shown in \cite{Hjo} that as a consequence the corresponding orbit equivalence relation on $X$ is not classifiable by countable structures.\\

Our aim is now to show that the orbit equivalence given by the action of $\Iso(\Ur)$ on  $\Rep(\Gamma,\Ur)$ by conjugation is generically turbulent whenever $\Gamma$ is infinite. By Theorem 3.21 in \cite{Hjo}, it is sufficient to prove that all equivalences classes are meager, which we have proved in Theorem \ref{thm_meagerclassesUrysohn}, and that there exists a turbulent element whose equivalence class is dense. That will be the content of the following theorem.

We note that Kerr, Li and Pichot \cite{KLP} prove analogous statements for unitary representations. Namely, in \cite[Theorem 3.3]{KLP} they prove that for any countable group $\Gamma$ the action of $U(\mathcal{H})$ on the space $\Rep_\lambda(\Gamma,\mathcal{H})$ of all unitary representations weakly contained in the regular representation is generically turbulent. It follows from \cite[Theorem 2.5]{KLP} that for any countable group without property (T) the action of $U(\mathcal{H})$ on $\Rep(\Gamma,\mathcal{H})$ is generically turbulent.

\begin{thm}\label{thm_turbulent_element}
Let $\Gamma$ be a countably infinite group. Then there exists $\alpha\in\Rep(\Gamma,\Ur)$ whose conjugacy class is dense and which is turbulent.
\end{thm}
\begin{proof}
First we work with finitely generated groups. Fix some finitely generated group $\Gamma$. Let $I$ be some finite index set and call a generalized pseudonorm $N$ on $\Gamma\times I^2$ \emph{finitely generated} if there exists a sufficient partial generalized pseudonorm with finite domain such that $N$ is its maximal extension (guaranteed by Fact \ref{fact_extension2}). An \emph{embedding} between two generalized pseudonorms $N_1$ on $\Gamma\times I_1^2$ and $N_2$ on $\Gamma\times I_2^2$ is an injection $\iota:I_1\hookrightarrow I_2$ such that for every $g\in \Gamma$ and $i,j\in I_1$ we have $N_1(g,i,j)=N_2(g,\iota(i),\iota(j))$. Typically, $\iota$ will be just the inclusion.
\begin{lem}
The class of all finitely generated rational valued generalized pseudonorms is a \fra class.
\end{lem}
Note that being rational valued means that it is the maximal extension of some partial generalized pseudonorm with finite domain which is rational valued.
\begin{proof}
This is straightforward, we only check the amalgamation property. Suppose we are given such generalized pseudonorms $N_1$, $N_2$ and $N_3$ defined on $\Gamma\times I_1$, $\Gamma\times I_2$ and $\Gamma\times I_3$ respectively, and $I_1\subseteq I_2\cap I_3$. That is, there are embeddings from $N_1$ into $N_2$ and $N_3$. Suppose that $N_i$ is the maximal extension of a partial generalized pseudonorm $P_i$ defined on $A_i\subseteq \Gamma\times I_i^2$, where $i\in\{2,3\}$. Set $I_4=I_2\cup I_3$ and $A_4=A_2\cup A_3$. Define a partial generalized pseudonorm $P_4$ on $A_4\subseteq \Gamma\times I_4^2$ so that it extends $P_2$, resp. $P_3$. Note that this is well-defined as $P_2$ and $P_3$ agree on $A_2\cap A_3$. Set $N_4$ to be the maximal extension of $P_4$. This is readily checked to be the desired amalgam.
\end{proof}
The \fra limit is some rational valued generalized pseudonorm $N_F$ defined on $\Gamma\times I_F^2$, for some infinite $I_F$. $N_F$ corresponds to an action of $\Gamma$ on some rational valued countable metric space with orbits indexed by $I_F$, and for each $i\in I_F$ there is a distinguished base point $x_i$ of that orbit. It is straightforward to check that this metric space is isometric to $\Rat\Ur$ and that $D=\{x_i:i\in I_F\}$ is dense. Denote by $\alpha$ this action on $\Rat\Ur$ and use the same letter also for its extension on $\Ur$, the completion of $\Rat\Ur$. We aim to show that $\alpha$ is the desired element from the statement of Theorem \ref{thm_turbulent_element}.

\begin{lem}
The conjugacy class of $\alpha$ is dense.
\end{lem}
\begin{proof}
Fix an open neighborhood $O$ of some $\beta\in\Rep(\Gamma,\Ur)$ given by some $x_1,\ldots,x_n\in \Ur$, some finite symmetric $F\subseteq \Gamma$ containing the unit and some $\varepsilon>0$. Set $F'=\{g^{-1}h:g,h\in F\}$. Set $I=\{1,\ldots,n\}$. Define a partial generalized pseudonorm $P$ on $F'\times I^2\subseteq \Gamma\times I^2$ as follows: set $P(g,i,j)=d_\Ur(x_i,g.x_j)$ for $(g,i,j)\in F'\times I^2$. By $\varepsilon$-perturbing $P$ a bit if necessary, we may assume that $P$ is rational valued. Now using the property of the \fra limits we see that that there are $i_1,\ldots,i_n\in I_F$ such that $N_F(g,i_k,i_l)=P(g,k,l)$ for every $g\in F'$ and $k,l\in I$. Clearly we can then find an isometry $\phi\in\Iso(\Ur)$ such that $\phi^{-1}\alpha\phi\in O$.
\end{proof}
 
 The rest of the proof will therefore focus on proving the following proposition.
\begin{prop}\label{prop:turbulent_element}
The element $\alpha$ is turbulent
\end{prop}
Suppose, without loss of generality, that $I_F=\Nat$. Fix some open neighborhood $U$ of $\alpha$ which we may assume is given by $x_1,\ldots,x_n\in D\subseteq \Rat\Ur$, finite symmetric $F\subseteq \Gamma$ containing the unit and some $\varepsilon>0$, and some open neighborhood $V$ of $\mathrm{id}$ in $\Iso(\Ur)$ which we may assume is given also by the same $x_1,\ldots,x_n\in D\subseteq \Rat\Ur$ and $\varepsilon>0$. Recall that by the construction, for every $f,g\in F$, $i,j\leq n$ we have $d_\Ur(\alpha(f).x_i,\alpha(g).x_j)=N_F(g^{-1}f,i,j)$. Set $F'=\{g^{-1}f:g,f\in F\}$ and $I=\{1,\ldots,n\}$. The restriction of $N_F$ onto $\Gamma\times I^2$ is a finitely generated generalized pseudonorm. Without loss of generality we may assume that it is generated by the values on $A=F'\times I^2$.

Now we check that the local orbit $O(U,V)$ of $\alpha$ is dense in $U$. Take some $\beta\in U$ and an open neighborhood $W\subseteq U$ of $\beta$ given by $x_1,\ldots,x_n,\ldots,x_m\in D\subseteq \Rat\Ur$, some finite symmetric $F\subseteq H\subseteq \Gamma$ and some $\varepsilon>\varepsilon'>0$. Set again $H'=\{g^{-1}f:g,f\in H\}$, $J=\{1,\ldots,m\}$ and let $P_\beta$ be the partial generalized pseudonorm on $H'\times J^2$ defined as $P_\beta(g,i,j)=d_\Ur(x_i,\beta(g).x_j)$. Again, by perturbing $P_\beta$ a bit if necessary, we may assume that $P_\beta$ is rational valued. Analogously, set $P_\alpha$ to be the restriction of $N_F$ onto $H'\times J^2$. We may suppose, without loss of generality, that the finitely generated generalized pseudonorm $N_F\upharpoonright \Gamma\times J^2$ is generated by its values on $H'\times J^2$, i.e. it is generated by $P_\alpha$.

The goal is now to find an isometry $\phi\in\Iso(\Ur)$ such that the partial generalized pseudonorm determined by the action $\phi\circ \alpha\circ \phi^{-1}$, the subset $H'\subseteq \Gamma$ and elements $\{x_1,\ldots,x_m\}$ is equal to $P_\beta$. By the definition of the local orbit $O(U,V)$, the desired isometry $\phi$ must be obtained as some product $\phi_n\circ\ldots\circ \phi_1$, where for each $i\leq n$, $\phi_i\in V$ and $\phi_i\circ\ldots\circ\phi_1\circ\alpha\circ\phi_1^{-1}\circ\ldots\circ\phi_i^{-1}\in U$. In order to do it, we shall find a `path of partial generalized pseudonorms $P_\alpha=P_0, P_1,\ldots,P_n=P_\beta$ on $H'\times J^2$ such that the neighbours on the path are close enough to each other that one can get from one to the other by conjugating with an isometry from $V$. The path will consist of convex combinations of $P_\alpha$ and $P_\beta$. It will be made precise below.
\medskip

Before we proceed further we need a few notions and basic lemmas.
\begin{defin}\label{defin:dist_on_pseudonorms}
Suppose that $N_1$ and $N_2$ are two partial generalized pseudonorms defined on the same set of the form $A\times L^2$ for some $A\subseteq \Gamma$. Then we define the distance $D(N_1,N_2)$ between them as the supremum distance, i.e. $\sup_{a\in A\times L^2} |N_1(a)-N_2(a)|$.

For any $n\in\Nat$ denote by $L\times n$ the set $\{(i,j):i\in L,j\leq n\}$ and for any $j\leq n$ denote by $L(j)\subseteq L\times n$ the subset $\{(i,j):i\in L\}$. Finally, denote by $(L\times n)'\subseteq (L\times n)^2$ the symmetric set $\bigcup_{i\in L,j<n}\{((i,j),(i,j+1))\}\cup\{((i,j+1),(i,j))\}$.
\end{defin}
\begin{lem}\label{lem_turbulent1}
Pick some partial generalized pseudonorms $N_1$, $N_2$ defined on the same set of the form $A\times L^2$ for some $A\subseteq \Gamma$, like in Definition~\ref{defin:dist_on_pseudonorms} above. Then there exists a partial generalized pseudonorm $N$ defined on $B=(A\times L(1)^2)\cup (A\times L(2)^2) \cup (\{1_\Gamma\}\times (L\times n)')$ such that
\begin{itemize}
\item $N(a,(i,k),(j,k))=N_k(a,i,j)$, for $k\in\{1,2\}$, $a\in A$ and $i,j\in L$;
\item $N(1_\Gamma, (i,j),(i,j+1))=D(N_1,N_2)$.

\end{itemize}
\end{lem}
\begin{proof}
Define $N$ as in the statement of the lemma. We must check that it is a partial generalized pseudonorm. The only non-trivial thing to check is the triangle inequality. Suppose that the triangle inequality does not hold. That is, there are $g,g_1,\ldots,g_j\in A$, $(i_1,i'_1),\ldots,(i_{j+1},i'_{j+1})\in L\times n$ such that
\begin{itemize}
\item $g=g_1\ldots g_j$;
\item $(g_l,(i_l,i'_l),(i_{l+1},i'_{l+1}))\in B$ for $l\leq j$ and $(g,(i_1,i'_1),(i_{j+1},i'_{j+1}))\in B$;
\item $N(g,(i_1,i'_1),(i_{j+1},i'_{j+1}))>\sum_{l=1}^j N(g_l,(i_l,i'_l),(i_{l+1},i'_{l+1}))$.

\end{itemize}
The case when $g=1_\Gamma$ and $i'_1\neq i'_{j+1}$, i.e. $(g,(i_1,i'_1),(i_{j+1},i'_{j+1}))\in\{1_\Gamma\}\times (L\times n)'$ is straightforward. So let us assume that $i'_1=i'_{j+1}=1$, the case when it is equal to $2$ is symmetric.

Suppose also that $j$ above is the least possible. We claim that for no $l<j$ we have $i'_l=i'_{l+1}=i'_{l+2}$. Indeed, by the triangle inequality we have $$N(g_l,(i_l,i'_l),(i_{l+1},i'_{l+1}))+N(g_{l+1},(i_{l+1},i'_{l+1}),(i_{l+2},i'_{l+2}))=$$ $$N_{i_l}(g_l,(i_l,i'_l),(i_{l+1},i'_{l+1}))+N_{i_l}(g_{l+1},(i_{l+1},i'_{l+1}),(i_{l+2},i'_{l+2}))\geq$$ $$N_{i_l}(g_l g_{l+1}, (i_l,i'_l),(i_{l+2},i'_{l+2}))=N(g_l g_{l+1}, (i_l,i'_l),(i_{l+2},i'_{l+2})).$$
So we may shorten the decomposition of $g$ which contradicts that $j$ was the least possible.

It follows that without loss of generality we may assume that for $l\leq j$ odd we have that $g_l\in A$ and $i'_l=i'_{l+1}$, while for $l\leq j$ even we have $g_l=1_\Gamma$ and $i'_l\neq i'_{l+1}$. Suppose that for some $l\leq j$ we have that $i'_l=2$. It follows that $l<j$ therefore by the paragraph above we have $$N(g_{l-1},(i_{l-1},i'_{l-1}),(i_l,i'_l))+N(g_l,(i_l,i'_l),(i_{l+1},i'_{l+1}))+$$ $$N(g_{l+1},(i_{l+1},i'_{l+1}),(i_{l+2},i'_{l+2}))=D(N_1,N_2)+N_2(g_l,i_l,i_{l+1})+D(N_1,N_2)\geq$$ $$N_1(g_l,i_l,i_{l+1}).$$
Therefore we may again shorten the decomposition which again contradicts that $j$ was the least possible. It follows that for no $l\leq j$ we have that $i'_l=2$. However, then we are clearly done.
\end{proof}
The proof of the following lemma is very easy and left to the reader. Note in particular that if $N$ is a generalized pseudonorm and $r>0$ is a positive real, then $rN$ is a generalized pseudonorm, and if $N_1$, $N_2$ are two generalized pseudonorms, then $N_1+N_2$ is as well.
\begin{lem}\label{lem_turbulent2}
Suppose that we have generalized pseudonorms $N_1$ and $N_2$ as above. Take any $0<t<1$. Then for the convex combination $N_3=tN_1+(1-t)N_2$ we have $D(N_1,N_3)=(1-t)D(N_1,N_2)$ and $D(N_2,N_3)=tD(N_1,N_2)$.
\end{lem}

We are now ready to finish the proof of Proposition \ref{prop:turbulent_element} and therefore the proof of Theorem \ref{thm_turbulent_element}.

Set $M=D(P_\alpha,P_\beta)$. Let $\delta>0$ be an arbitrary rational number such that $\delta<\varepsilon$ and $k=M/\delta+1$ is in $\Nat$. For every $1\leq j< k$ let $P_j$ be the convex combination $\frac{k-j}{k}P_\alpha+\frac{j}{k}P_\beta$ defined on $H'\times J^2$. Notice that since for every $i,i'\in J$ we have $P_\alpha(1_\Gamma,i,i')=P_\beta(1_\Gamma,i,i')$, we also have for every $1\leq j<k$, $P_\alpha(1_\Gamma,i,i')=P_j(1_\Gamma,i,i')$.  Moreover, since for every $i,i'\in I$ and every $g\in F'$ we have $|P_\alpha(g,i,i')-P_\beta(g,i,i')|<\varepsilon$, we also have for every $1\leq j<k$, $|P_\alpha(g,i,i')-P_j(g,i,i')|<\varepsilon$. Finally, for any $0\leq j<k$ (where we identify $0$ with $\alpha$) and for any $i,i'\in J$, $g\in H'$ we have $|P_j(g,i,i')-P_\beta(g,i,i')|\leq M-j\cdot\delta$.\\

Now we define a rational valued sufficient partial generalized pseudonorm $N$ on $C=(\bigcup_{j\leq k} H'\times J(j)^2)\cup (\{1_\Gamma\}\times (J\times k)')$ as follows:
\begin{itemize}
\item For any $b\in \{1_\Gamma\}\times (J\times k)'$ we set $N(b)=\delta$.
\item For any $g\in H'$ and $i,i'\leq m$ we set $N(g,(i,1),(i',1))=P_\alpha(g,i,i')$ and  $N(g,(i,k),(i',k))=P_\beta(g,i,i')$.
\item For any $g\in H'$, $i,i'\leq m$ and $1<j<k$ we set $N(g,(i,j),(i',j))=P_j(g,i,i')$.

\end{itemize}
It follows from Lemma \ref{lem_turbulent2} and repeated use of Lemma \ref{lem_turbulent1} that $N$ is indeed a rational valued sufficient generalized pseudonorm.\\

By the extension property of the \fra limit $N_F$ we may realize $N$ as the extension of $N_F\upharpoonright \Gamma\times I^2$ and as a subfunction of $N_F$. To simplify the notation, we view $C$ as a subset of $\Gamma\times I_F^2$ and assume that $N_F\upharpoonright C=N\upharpoonright C$. Notice that each element of $C$ is of the form $(g,(i,j),(i',j'))$. We shall therefore denote the elements of $D$ corresponding to these pairs (e.g. $(i,j)$) of indices by $x_{(i,j)}$. With this identification, the original elements $x_1,\ldots,x_m\in D$ will now correspond to elements $x_{(1,1)},\ldots,x_{(m,1)}$ as the index set $J=\{1,\ldots,m\}$ has been identified with the set $J(1)$.

By the standard arguments using the homogeneity of the Urysohn space, we can find isometries $\phi_1,\ldots,\phi_k\in\Iso(\Ur)$ such that:
\begin{itemize}
    \item For each $l\leq k$, we have $\phi_l(x_{(i,j)})=x_{(i,j-1)})$, for every $i\leq m$ and $1<j\leq k$ (the subspaces $\{x_{(i,j)}:i\leq m\}$ and $\{x_{(i,j-1)}:i\leq m\}$ are isometric). This in particular implies that $\phi_l\in V$, for every $l\leq k$.
    \item For every $l\leq k$, every $i\leq n$ and $f\in F$ we have $d_\Ur(\alpha(f).x_i,\phi_l\circ\ldots\circ\phi_1\circ\alpha\circ\ldots\circ\phi_l^{-1}(f).x_i)<\varepsilon$. This implies that for every $l\leq k$ we have $\phi_l\circ\ldots\circ\phi_1\circ\alpha\circ\ldots\circ\phi_l^{-1}\in U$.
    \item For every $l\leq k$, every $i\leq m$ and $h\in H$ we have $d_\Ur(\alpha(h).x_i,\phi_l\circ\ldots\circ\phi_1\circ\alpha\circ\ldots\circ\phi_l^{-1}(h).x_i)<M-l\cdot\delta$. This implies that $\phi_k\circ\ldots\circ\phi_1\circ\alpha\circ\ldots\circ\phi_k^{-1}\in W$.
\end{itemize}
This finishes the proof of Proposition \ref{prop:turbulent_element} and also Theorem \ref{thm_turbulent_element} for finitely generated groups.\\

If $\Gamma$ is infinitely generated, then write $\Gamma$ as an increasing union $\Gamma_1\leq \Gamma_2\leq\ldots$ of finitely generated subgroups. Denote by $\alpha_n$, for $n\in\Nat$, the action of $\Gamma_n$ on $\Rat\Ur$ obtained as a \fra limit as above. Since the \fra limits are uniquely characterized by the extension property, it is easy to check that for every $n\in\Nat$ the restriction of $\alpha_{n+1}$ onto $\Gamma_n$ is isomorphic to $\alpha_n$. Denote by $U_n$ the copy of $\Rat\Ur$ on which $\alpha_n$ acts. By the identification of $\alpha_n$ with the restriction of $\alpha_{n+1}$ we may view $U_n$ as a subspace of $U_{n+1}$. Moreover, it is easy to see that $U_n$ is dense in $U_{n+1}$. Therefore we get the direct limit of the actions $\alpha_n$ to be an action $\alpha$ on the union $U=\bigcup_n U_n$ which is also isometric to $\Rat\Ur$ and each $U_n$ is dense in $U$. Therefore $\alpha$ is also naturally a direct limit of the actions $\alpha_n$ on the completion $\Ur$ and it follows from the argument above, for finitely generated groups, that $\alpha$ is a turbulent element in $\Rep(\Gamma,\Ur)$.
\end{proof}
Let us mention that an analogous proof can be also used to show that the conjugacy action of $\Iso(\Ur_1)$ on $\Rep(\Gamma,\Ur_1)$, for any infinite $\Gamma$, is generically turbulent. The turbulent element there is the completion of the \fra limit of finitely generated \emph{bounded} rational generalized pseudonorms on $\Gamma$. 
\begin{thm}
The conjugacy action of $\Iso(\Ur_1)$ on $\Rep(\Gamma,\Ur_1)$ is generically turbulent for any countably infinite group $\Gamma$.
\end{thm}

\end{document}